\documentclass[11pt]{article} 
\usepackage{caption}
\usepackage[utf8]{inputenc} 
\usepackage{amsmath}
\usepackage{amssymb}
\usepackage{floatpag}

\usepackage{ulem}

\usepackage{amsthm}
\usepackage{ mathrsfs }
\usepackage{placeins}
\usepackage{titlesec}
\titleformat{\paragraph}
{\normalsize\bfseries}{\theparagraph}{1em}{}
\titlespacing*{\paragraph}
{0pt}{3.25ex plus 1ex minus .2ex}{1.5ex plus .2ex}

\titleclass{\subsubsubsection}{straight}[\subsection]

\newcounter{subsubsubsection}[subsubsection]
\renewcommand\thesubsubsubsection{\thesubsubsection.\arabic{subsubsubsection}}
\renewcommand\theparagraph{\thesubsubsubsection.\arabic{paragraph}} 

\titleformat{\subsubsubsection}
  {\sffamily\small}{\thesubsubsubsection}{1em}{}
\titlespacing*{\subsubsubsection}
{0pt}{3.25ex plus 1ex minus .2ex}{1.5ex plus .2ex}

\newcommand{\RN}[1]{%
  \textup{\uppercase\expandafter{\romannumeral#1}}%
}

\newtheorem{Theorem}{Theorem}[section]
\newtheorem{Corollary}{Corollary}[section]
\newtheorem{Lemma}{Lemma}[section]
\newtheorem{Definition}{Definition}[section]
\newtheorem{Remark}{Remark}[section]
\newtheorem{Notation}{Notation}[section]
\usepackage[table]{xcolor}
\usepackage{float}
\restylefloat{table}
\usepackage{dblfloatfix}
\usepackage[hidelinks]{hyperref}

\usepackage[hidelinks]{hyperref}
\usepackage{xcolor}
\hypersetup{
    colorlinks,
    linkcolor={blue!50!black},
    citecolor={blue!50!black},
    urlcolor={blue!80!black}
}

\usepackage{placeins}
  \usepackage{graphicx}
    \usepackage{caption}    
    \usepackage{array,booktabs} \usepackage{amsthm,amsmath,amsfonts,amssymb,latexsym,amscd,mathrsfs}
    \usepackage{textcomp,float}
    \usepackage[margin=25mm]{geometry}

    \usepackage{tikz}
    \usepackage{tikz-3dplot}
    \usetikzlibrary{shapes,calc,positioning}
\usepackage{afterpage}
\usepackage{placeins}
\usepackage{hyperref}

\usepackage{geometry} 
\definecolor{thistle}{rgb}{0.85, 0.75, 0.85}
\definecolor{blue-green}{rgb}{0.0, 0.87, 0.87}
\def\bF{\mathbb{F}}
\def\Fq{\mathbb{F}_q}
\def\GL{\mathrm{GL}}
\def\PG{\mathrm{PG}}
\def\PGL{\mathrm{PGL}}
\usepackage{graphicx} 
\definecolor{amber(sae/ece)}{rgb}{1.0, 0.49, 0.0}
	\definecolor{darkcyan}{rgb}{0.0, 0.55, 0.55}
	\definecolor{darkseagreen}{rgb}{0.56, 0.74, 0.56}
	\definecolor{salmon}{rgb}{1.0, 0.55, 0.41}
\usepackage{booktabs} 
\usepackage{array} 
\usepackage{paralist} 
\usepackage{verbatim} 
\usepackage{subfig} 

\usepackage{fancyhdr} 
\definecolor{carminepink}{rgb}{0.92, 0.3, 0.26}
	\definecolor{electricyellow}{rgb}{1.0, 1.0, 0.0}
	\definecolor{chromeyellow}{rgb}{1.0, 0.65, 0.0}
	\definecolor{babyblue}{rgb}{0.54, 0.81, 0.94}

\usepackage{sectsty}
\allsectionsfont{\sffamily\mdseries\upshape} 

 \def\bF{\mathbb{F}}
\def\Fq{\mathbb{F}_q}
\def\GL{\mathrm{GL}}

\def\F2h{\mathbb{F}_{2^h}}

\def\PG{\mathrm{PG}}
\def\PGL{\mathrm{PGL}}

\def\cL{\mathcal{L}}
\def\cV{\mathcal{V}}
\def\cZ{\mathcal{Z}}
\def\cC{\mathcal{C}}
\def\cH{\mathcal{H}}

\def\cH{\mathcal{H}}

\usepackage[nottoc,notlof,notlot]{tocbibind} 
\usepackage[titles,subfigure]{tocloft}

\title{A classification of planes intersecting the Veronese surface over finite fields of even order}

\usepackage{placeins}

\author{Nour Alnajjarine\footnote{Sabanc\i~University, Istanbul, Turkey; \texttt{nour@sabanciuniv.edu}} \;  and \;
Michel Lavrauw\footnote{Sabanc\i~University, Istanbul, Turkey; \texttt{mlavrauw@sabanciuniv.edu}} }

\begin{document}
\maketitle
\begin{abstract}

In this paper we contribute towards the classification of partially symmetric tensors in $\bF_q^3\otimes S^2\bF_q^3$, $q$ even, by classifying planes which intersect the Veronese surface $\cV(\Fq)$ in at least one point, under the action of $K\leq \PGL(6,q)$, $K\cong \PGL(3,q)$, stabilising the Veronese surface. We also determine a complete set of geometric and combinatorial invariants for each of the orbits.\\\par
\textbf{Keywords:}\hspace{0.2cm}Veronese Surface,\hspace{0.2cm} Tensors,\hspace{0.2cm} Ranks,\hspace{0.2cm} Nets of Conics,\hspace{0.2cm}  Finite Fields.

\end{abstract}

\section{Introduction}

Denote by $W = S^n \mathbb{F}^m$ the subspace of symmetric tensors in the $n$-fold tensor product space $V=\mathbb{F}^m\otimes...\otimes \mathbb{F}^m$ defined over a field $\mathbb{F}$. 
Decomposing tensors is a notoriously difficult problem, and even more so is classifying tensors up to (natural) equivalence. 

A tensor  $A\in \bF^3\otimes \bF^3\otimes \bF^3$ has a first contraction space  $C_1(A)\leq \bF^3\otimes \bF^3$, and $A$ is called {\it partially symmetric} if $C_1(A)\leq S^2\bF^3$, in other words $A\in \bF^3\otimes S^2\bF^3$.
In this paper we classify certain types of partially symmetric tensors in $\bF^3\otimes S^2\bF^3$ over finite fields $\bF$ of even order. The original motivation for studying 3-fold tensor products over finite fields comes from the theory of semifields and non-Desarguesian projective planes, see \cite{Lavrauw2011}.
The geometric approach which is used in this paper was originated in previous work by the second author and J. Sheekey \cite{canonical}, in which tensors in $\bF^2\otimes \bF^3\otimes \bF^3$ were classified, and by the second author and T. Popiel \cite{lines}, in which it was determined which line-orbits in $\PG(\bF^3\otimes \bF^3)$, arising from this classification, have a symmetric representation. Although a complete classification of tensors in $\bF^3\otimes \bF^3\otimes \bF^3$ seems, at this point in time, out of reach, some progress has been made by reversing the previous approach, i.e. by starting with the study of partially symmetric representations of tensors in $\bF^3\otimes \bF^3\otimes \bF^3$. Classifying such tensors is equivalent to classifying orbits of subspaces of $\PG(5,q)$ under the group $K\cong \PGL(3,q)$ leaving the Veronese surface $\cV(\bF_q)$ invariant. The action of $K$ can be obtained by lifting the natural action of $\PGL(3,q)$ on $\PG(2,q)$ through the Veronese map from $\PG(2,q)$ to $\PG(5,q)$ (see e.g. Harris \cite[Lecture 10]{harris}). It is known that $K$ is the setwise stabiliser of $\cV(\bF_q)$ in $\PGL(6,q)$ unless $q=2$. 
If $q=2$, then $\PGL(3,2)$ is strictly contained in the setwise stabiliser of $\cV(\mathbb{F}_2)$ equivalent to $\text{Sym}_7$. 
\par 

For all $q$, points, lines, solids and hyperplanes are completely classified in $\PG(5,q)$. Indeed, $K$-orbits of points and hyperplanes are easily determined (see Section \ref{pre}), lines were classified in \cite{lines}, and solids in \cite{solidsqeven}. Moreover, planes in $\PG(5,q)$ intersecting the Veronese surface $\cV(\Fq)$ in at least one point, $q$ odd, are classified in \cite{nets,LaPoSh2021}.

Therefore, to  complete the
classification of subspaces of $\PG(5,q)$, we need to classify planes which are disjoint from $\cV(\Fq)$ for all $q$ and planes meeting $\cV(\Fq)$ for $q$ even. In this paper, we investigate the latter classification by considering cubic curves associated with planes in $\PG(5,q)$. In
particular, we compute for each (type of) plane $\pi \subseteq \PG(5,q)$ its point-orbit distribution represented by the $4$-tuple $[r_1,r_{2n},r_{2s},r_3]$, where $r_i$ is the number of rank-$i$ points in $\pi$ for $i \in\{1,3\}$, $r_{2n}$ is the number of rank-$2$ points in $\pi$ meeting the nucleus plane $\mathcal{N}$ and $r_{2s}$ is the number of the remaining rank-$2$ points in $\pi$. In general, we distinguish between orbits using point-orbit distributions, line-orbit distributions and inflexion points defined in Section  \ref{pre}. While some of the arguments used in our classification come from the classification of planes meeting $\cV(\Fq)$, $q$ odd  \cite{nets}, a unified proof for all $q$ seems impossible due to the fundamental differences in the geometry of the Veronese surface in even characteristic.

Historically, the study of subspaces in the space of the Veronese surface began more than a century ago, through the study of linear systems of conics, which can also be interpreted as the classification of $K$-orbits of subspaces of $\PG(5,\bF)$.
More precisely, classifying solids, planes and lines of $\PG(5,\mathbb{F})$  correspond to classifying {\it pencils}, {\it nets} and {\it webs} of conics, namely $1$-, $2$- and $3$-dimensional systems, respectively. This problem was completely solved over $\mathbb{R}$ and $\mathbb{C}$ by Jordan and Wall who classified pencils and nets in \cite{jordan1,jordan2,wall}. We also refer to \cite{AbEmIa}, which includes an interesting historical account on nets of conics in the appendices A and B.

For finite fields, the narrative is more complicated. For pencils of conics, first studied by Dickson for $q$ odd, and later by Campbell for $q$ even, we refer to \cite{lines} and \cite{solidsqeven} for historical comments, and a discussion of the results. 
Nets over finite fields of odd order, were studied in 1914 by Wilson \cite{wilson}. For $q$ even, these nets were studied in 1928 by Campbell \cite{campbell2}. While both works definitely have their merits non of these gives a complete classification of nets of conics. For the odd characteristic case, we refer to \cite[Section 9]{nets} for an explanation of some of the shortcomings of Wilson's treatment. For the even characteristic case Zanella \cite{zanella} already pointed out an error in Campbell's work, by providing a construction of nets with $q^2+q+1$ non-singular conics for all $q$, disproving a statement in \cite{campbell2} that such nets only exist when $q \equiv 1\:\:(mod\:\: 3)$. Also, some of the arguments in \cite{campbell2} (such as \cite[pp. 482]{campbell2}) are based on Campbell's claim (see \cite{campbell}) that pencils with $q$ non-singular conics and a unique pair of conjugate imaginary lines do not exist, which is incorrect, see \cite[Section 7]{solidsqeven}. Another issue with Campbell's paper (and there are more) is the fact that some of the representatives of his “classes" actually represent non-equivalent nets, depending on a parameter. An example of this is the net represented by equation (32) on page 489 of \cite{campbell2}. As a general rule, when referring to \cite{campbell2}, one should keep in mind that it distinguishes a number of non-equivalent nets of conics, but it certainly does not determine the actual equivalence classes.

\par 

The paper is structured as follows. The proof of our main result, Theorem \ref{mainTheoremplanes}, is given in Sections~\ref{atleast3}--\ref{sigma14primeorbit}. In particular, we consider in Section \ref{atleast3} planes intersecting the Veronese surface $\cV(\Fq)$ in at least three points. In Section \ref{2rankone}, we classify  planes meeting $\cV(\Fq)$ in exactly two points. In Section \ref{qwe}, we discuss planes spanned by points of rank at most $2$ which meet $\cV(\Fq)$ in a unique point. Finally, we investigate in Section \ref{sigma14primeorbit} planes which are not spanned by points of rank at most $2$ and meet $\cV(\Fq)$ in a unique point. Note that, the case $q=2$ requires special treatment, and is handled in Section~\ref{q2planes}. We finish in Section \ref{comparison22} by a comparison with the similar classification over finite fields of odd characteristic. In particular, we prove that the correspondence, between nets of rank $1$ and planes meeting the Veronese surface, which holds for $q$ odd, fails when $q$ is even.

\begin{Theorem} \label{mainTheoremplanes} 
Let $q$ be an even prime power. 
There are exactly 15 orbits of planes having at least one rank-1 point in $\PG(5,q)$ under the induced action of $\PGL(3,q) \leqslant \PGL(6,q)$ defined in Section \ref{pre}. 
Representatives of these orbits are given in Table~\ref{tableplanes}, the notation of which is defined in Section \ref{pre}. 
\end{Theorem}

 \floatpagestyle{empty} 

\begin{table}[!htbp]
\begin{center}
\footnotesize{

\begin{tabular}[!htbp]{ l ll l} 
 \hline
$K$-orbits of planes & Representatives   &Point-orbit distributions &Conditions\\ \hline
\vspace{0.2cm}

$ \Sigma_1$& $\begin{bmatrix} x&y&.\\y&z&.\\.&.&.          \end{bmatrix}$ & $[q+1,1,q^2-1,0]$&  
 \\ \vspace{0.2cm}

$ \Sigma_2$&	$\begin{bmatrix} x&.&.\\.&y&.\\.&.&z          \end{bmatrix}$ & 
$[3,0,3q-3,q^2-2q+1]$&
\\ \vspace{0.2cm}

$\Sigma_{3}$  &$\begin{bmatrix} x&.&z\\.&y&.\\z&.&.        \end{bmatrix}$ &$[2,1,2q-2,q^2-q]$ &
 \\\vspace{0.2cm}
$\Sigma_{4}$  &$\begin{bmatrix} x&.&z\\.&y&z\\z&z&.        \end{bmatrix}$ &$[2,1,2q-2,q^2-q]$ &
\\ \vspace{0.2cm}

$\Sigma_5 $  & $\begin{bmatrix} x&.&z\\.&y&z\\z&z&z       \end{bmatrix}$ &  $[2,0,2q-2,q^2-q+1]$&
\\\vspace{0.2cm}

       $\Sigma_{6}$ & $\begin{bmatrix} x&.&.\\.&y+c z&z\\.&z&y        \end{bmatrix}$  & $[1,0,q+1,q^2-1]$&
       $Tr(c^{-1})=1$ \\\vspace{0.2cm}

  \vspace{0.2cm}

  $ \Sigma_7$&	$\begin{bmatrix} x&y&z\\y&.&.\\z&.&.          \end{bmatrix}$ & $[1,q+1,q^2-1,0]$&
   \\ \vspace{0.2cm}
   $ \Sigma_8$&	$\begin{bmatrix} x&y&.\\y&.&z\\.&z&.          \end{bmatrix}$ & $[1,q+1,q-1,q^2-q]$& 
   \\ \vspace{0.2cm}
   $ \Sigma_9$&	$\begin{bmatrix} x&y&.\\y&z&z\\.&z&.          \end{bmatrix}$ & $[1,1,2q-1,q^2-q]$&
   \\ \vspace{0.2cm}

 $\Sigma_{10}$ & $\begin{bmatrix} x&y&.\\y&z&.\\.&.&z          \end{bmatrix}$ &$[1,1,2q-1,q^2-q]$ &
 \\\vspace{0.2cm}

    $ \Sigma_{11}$& $\begin{bmatrix} x&y&.\\y&z&z\\.&z&x+z         \end{bmatrix}$	
     & $[1,1,q-1,q^2]$&
     \\ \vspace{0.2cm}
     
  $\Sigma_{12}$ &$\begin{bmatrix} x&y&c x\\y&y+z&.\\c x&.&c^2x+z \end{bmatrix}$  &$[1,0,q+1,q^2-1]$ &$Tr(c)=1 $,  $(*)$\\
  \vspace{0.2cm}

     $\Sigma_{13}$ & $\begin{bmatrix} x&y&cx\\y&y+z&.\\c x&.&c^2x+z \end{bmatrix}$ &$[1,0,q-1,q^2+1]$ &$Tr(c)=0$,  $(**)$ \\\vspace{0.2cm}

     $\Sigma_{14}$ & $\begin{bmatrix} x&y&c x\\y&y+z&.\\c x&.&c^2x+z \end{bmatrix}$ &$[1,0,q\mp 1,q^2\pm1]$ &$Tr(c)=Tr(1)$, $q\neq 4$, $(***)$ \\\vspace{0.2cm}

     $\Sigma_{14}'$ & $\begin{bmatrix} x+z&z&z\\z&y+z&z\\ z&z&y \end{bmatrix}$ &$[1,0,q- 1,q^2+1]$ &$q=4$ \\\vspace{0.2cm}

    $ \Sigma_{15}$&	$\begin{bmatrix} x&y&z\\y&z&.\\z&.&.         \end{bmatrix}$ & $[1,1,q-1,q^2]$&
    \\  \hline

   \end{tabular}}

 \caption{\label{tableplanes}$K$-orbits of planes in $\PG(5,q)$ meeting $\cV(\Fq)$ in at least one point and their point-orbit distributions, where $q\neq 2$ and $c$ is: $(*)$ not admissible if $q=2^{2m+1}$,  $(**)$ not admissible if $q=2^{2m}$ and $(***)$ admissible if $q>4$.  The point-orbit distribution in $\Sigma_{14}$ is given with respect to $q=2^{2m}$ and $q=2^{2m+1}$ respectively.    }

\end{center}
\end{table}

\section{Preliminaries}\label{pre}
In this section, we review some definitions and theory needed in our proofs, most of which can be found in \cite{berlekamp,glynn,harris,hirsch,lines}. 
We also refer the reader to Section 4.6 in \cite{galois geometry} for an overview of the interesting properties of the Veronese surface over finite fields. 
The {\it Veronese surface} $\cV(\bF_q)$ is a 2-dimensional algebraic variety in $\PG(5,q)$ defined as the image of the Veronese map
\[
\nu: \PG(2,q)\rightarrow \PG(5,q) \quad \text{given by} \quad
(u_0,u_1,u_2) \mapsto (u_0^2,u_0u_1,u_0u_2,u_1^2,u_1u_2, u_2^2).
\]
Alternatively, $\cV(\Fq)$ is the set of points  $(u_{00},u_{01},u_{02},u_{11},u_{12},u_{22})\in \PG(5,q)$ for which the rank of the symmetric matrix defined by $[u_{ij}]$ is $1$. 
We use the same notation as in \cite{nets,LaPoSh2021}, where a point $P=(y_0,y_1,y_2,y_3,y_4,y_5)$ of $\PG(5,q)$ is represented by a symmetric $3 \times 3$ matrix
\[
M_P=\begin{bmatrix} y_0&y_1&y_2\\
y_1&y_3&y_4\\
y_2&y_4&y_5  \end{bmatrix},
\]
and in which a dot represents a zero.
This representation can be extended to any subspace of $\PG(5,q)$. 
For example, the plane spanned by the first three points of the standard frame of $\PG(5,q)$ is represented by 
\begin{equation} \label{egSolid}
\begin{bmatrix} x&y&z\\y&\cdot&\cdot\\z&\cdot&\cdot \end{bmatrix}:=\left\{
\begin{bmatrix} x&y&z\\y&0&0\\z&0&0\end{bmatrix}:(x,y,z)\in \PG(2,q)
\right\}.
\end{equation}
The {\it rank} of a point $P\in\PG(5,q)$ is defined as  the rank of its associated matrix $M_P$. 
The points of rank $1$ are precisely those belonging to $\mathcal{V}(\Fq)$. 
Setting the determinant of the matrix representation of a plane $\pi$ in $\PG(5,q)$ equal to zero, defines a cubic curve $\mathscr{C}(\pi)$ in $\PG(2,q)$. The $\bF_q$-rational points on $\mathscr{C}(\pi)$ correspond to points of rank at most $2$ in $\pi$. 
Throughout the paper, the indeterminates of the coordinate rings of $\PG(2,q)$ and $\PG(5,q)$ are denoted by $(X,Y,Z)$ and $(Y_0,\ldots,Y_5)$ respectively, and $\cZ(f)$ denotes the zero locus of a form $f$. 

\begin{Notation}\label{groups}
We denote by $C_k$, $\operatorname{Sym}_k$ and $E_q$, a cyclic group of order $k$, a symmetric group on $k$ letters and, an elementary abelian group of order~$q$, respectively. Moreover, $A\rtimes B$ is a semidirect product with normal subgroup $A$ and subgroup $B$.
\end{Notation}

\subsection{Quadratic and cubic equations}

Over finite fields of characteristic 2, solutions of quadratic and cubic equations depend on the trace map $\operatorname{Tr}~:~\mathbb{F}_{2^h}\rightarrow \mathbb{F}_{2}$ defined by $$\operatorname{Tr}(x)=x+ x^2+x^{2^2}+...+x^{2^{h-1}}.
$$  
We will make frequent use of the following well-known lemma, see e.g. \cite{roots}.
\begin{Lemma}\label{quadratic}
The polynomial $f(X)=\alpha X^2 + \beta X + \gamma \in \mathbb{F}_{2^h}[X]$ with $\alpha \neq 0$ has exactly one root in $\mathbb{F}_{2^h}$ if and only if $\beta = 0$, two distinct roots in $\mathbb{F}_{2^h}$ if and only if $\beta \neq 0$ and $\operatorname{Tr}(\frac{\alpha \gamma}{\beta^2})=0$, and no roots in $\mathbb{F}_{2^h}$ otherwise.
\end{Lemma}

\begin{Remark} Solutions to a cubic equation $c(X)=0$ defined by $$c(X)=X^3+a_1X^2+a_2X+a_3\in \F2h[X]$$ can be retrieved by solving a cubic equation of the form  \begin{equation}g(\theta)=\theta^3+b\theta+a=0.\end{equation}
For instance, if $a_2\neq a_1^2$, we can apply the substitution $$X=(a_2+a_1^2)^{\frac{1}{2}} \theta+a_1,$$ to obtain the cubic polynomial $g(\theta)=\theta^3+\theta+a$ and $$a=\frac{a_3+a_2a_1}{(a_2+a_1^2)^{\frac{3}{2}} }.$$
Notice that, as the product of the three roots of $g(\theta)$ is equal to $a\in \F2h$, it follows that $g(\theta)$ is either irreducible over $ \F2h$, has all its roots in $ \F2h$, or exactly one root in $ \F2h$. 
 \end{Remark}
 
For cubic equations, we will make use of the next theorem, from \cite{berlekamp}, which gives the number of solutions in $\bF_{q}$, $q=2^h>2$, of $g(\theta)=\theta^3+\theta+a=0$ in terms of a condition on the constant term $a \in \bF_q$, which is called {\it admissible} if 
$$a=\frac{v+v^{-1}}{(1+v+v^{-1})^3}$$ for some $v \in \Fq\setminus \mathbb{F}_4$.

\begin{Theorem}\label{cubicequation}
Let $q=2^h>2$ and $a \in \bF_q\setminus\{0\}$. The cubic equation $\theta^3+\theta+a=0$ has
\begin{itemize}
\item three solutions in $\Fq$ if and only if $q\neq 4$, $\operatorname{Tr}(a^{-1})=\operatorname{Tr}(1)$ and $a$ is admissible, 
\item a unique solution in $\Fq$ if and only if $\operatorname{Tr}(a^{-1})\neq \operatorname{Tr}(1)$,
\item no solutions in $\Fq$ if and only if $\operatorname{Tr}(a^{-1})=\operatorname{Tr}(1)$ and $a$ is not admissible.
\end{itemize}
\end{Theorem}

\subsection{The Veronese surface in $\PG(5,q)$}

The Veronese surface $\cV(\Fq)$ contains $q^2 + q + 1$ conics, defined as the image of lines in $\PG(2,q)$ under the Veronese map $\nu$, where any two points $P$, $Q$ of $\mathcal{V}(\Fq)$ lie on a unique such conic given by 
\[
\mathcal{C}(P,Q):= \nu(\langle\nu^{-1}(P),\nu^{-1}(Q)\rangle).
\] 

 Since conics of $\cV(\Fq)$ correspond to lines of $\PG(2, q)$ via $\nu$, any two of these conics have a unique point in common. Also, the conics of $\PG(2, q)$ correspond to the hyperplane sections of $\cV(\Fq)$. If $\cC$ is a repeated line, then the corresponding hyperplane of $\PG(5, q)$ meets $\cV(\Fq)$ in a conic, if $\cC$ is a pair of real lines, then the corresponding hyperplane meets $\cV(\Fq)$ in two conics, if $\cC$ is a pair of conjugate imaginary lines, then the corresponding hyperplane meets $\cV(\Fq)$ in a point, if $\cC$ is a non-singular conic, then the corresponding hyperplane meets $\cV(\Fq)$ in a {\it normal rational curve}. For $q\neq 2$,  planes of $\PG(5, q)$ which meet $\cV(\Fq)$ in a conic are called {\it conic planes}.\par
 If $q$ is even, then all tangent lines to a conic $\mathcal{C}$ in $\mathcal{V}(\Fq)$ are concurrent, meeting at the {\it nucleus} of $\cC$. 
The set of all nuclei of conics in $\mathcal{V}(\Fq)$ coincides with the set of points of a plane in $\PG(5,q)$ known as the {\it nucleus plane} of $\cV(\bF_q)$ which is defined as the subspace $\mathcal{N}=\mathcal{Z}(Y_0, Y_3, Y_5)$ of
$\PG(5, q)$. In the matrix representation, points contained in the nucleus plane correspond to symmetric $3\times 3$ matrices with zeros on the main diagonal. 
Every rank-$2$ point $R$ of $\PG(5,q)$ defines a unique conic $\mathcal{C}(R)$ in $\cV(\bF_q)$. 
If $R$ lies on the secant $\langle P,Q \rangle$ with $P,Q \in \mathcal{V}(\Fq)$ then $\cC(R)=\cC(P,Q)$. 
If $R$ is contained in the nucleus plane, then $\cC(R)$ is the unique conic  contained in $\mathcal{V}(\Fq)$ with nucleus $R$.\par

Tangent lines of $\cV(\Fq)$ are tangent lines to the conics in $\cV(\Fq)$. 
Since $\cV(\Fq)$ is a non-singular 2-dimensional variety, all tangent lines of $\cV(\Fq)$ at a point $P\in \cV(\Fq)$ are contained in {\it the tangent plane of $\cV(\Fq)$ at $P$}.

\begin{Lemma}\label{lem:nuclei_pencil}
Let $\cL$ be a pencil of lines in $\PG(2,q)$, $q$ even, and denote by $\nu(\cL)$ the set of corresponding conics on $\cV(\bF_q)$ under the Veronese map. The set of nuclei of the conics in $\nu(\cL)$ is a set of $q+1$ collinear points in the nucleus plane of $\cV(\bF_q)$.
\end{Lemma}
\begin{proof}
By transitivity, it suffices to compute the nuclei of the pencil $\cL$ of lines through the point with coordinates $(1,0,0)$. The nuclei of the conics in $\nu(\cL)$ are the $q+1$ points of the line $\cZ(Y_0,Y_3,Y_4,Y_5)$ in the nucleus plane of $\cV(\bF_q)$.
\end{proof}

\subsection{The group action}
We are interested in the action on subspaces of $\PG(5,q)$ of the group $K\leqslant \PGL(6,q)$ defined as the lift of $\PGL(3,q)$ through the Veronese map $\nu$. 
Explicitly, if $\phi_A\in \PGL(3,q)$ is represented by the matrix $A \in \GL(3,q)$ then we define the corresponding projectivity $\alpha(\phi_A)\in \PGL(6,q)$ through its action on the points of $\PG(5,q)$ by 
\[
\alpha(\phi_A): P\mapsto Q \quad \text{where} \quad M_Q=AM_PA^T,
\] 
where $M_Q$ and $M_P$ are the matrix representations of $Q$ and $P$. Then $K:=\alpha(\PGL(3,q))$ is isomorphic to $\PGL(3,q)$ and leaves $\cV(\bF_q)$ invariant. 

\begin{Remark}\label{remq=2}
\textnormal{
If $q > 2$ then $K \cong \PGL(3,q)$ is the {\it full} setwise stabiliser of $\cV(\bF_q)$ in $\PGL(6,q)$. 
If $q=2$ then the full setwise stabiliser of $\cV(\bF_q)$ is $\text{Sym}_7$.
}
\end{Remark}

\subsection{Points, lines and hyperplanes of $\PG(5,q)$}

 The rank distribution of a subspace $W$ of $\PG(5,q)$ is the $3$-tuple $[r_1,r_2,r_3]$, where $r_i$ is the number of rank-$i$ points in $W$. Clearly, rank-distributions are $K$-invariants. The next definition introduces stronger $K$-invariants, first introduced in \cite{LaPoSh2021} and more formally in \cite{solidsqeven}, which play a fundamental role in classifying subspaces of $\PG(5,q)$.
 
 \begin{Definition}\label{orbitdist}
\textnormal{
Let $W_1,W_2,\dots,W_m$ denote a chosen ordering of the distinct $K$-orbits of $r$-spaces in $\PG(n,q)$. 
The \textit{$r$-space orbit distribution} of a subspace $W$ of $\PG(n,q)$ is the list $[w_1,w_2,\ldots,w_m]$, where $w_i$ is the number of elements of $W_i$ incident with $W$.
}
\end{Definition}
There are four $K$-orbits of points in $\PG(5,q)$: the orbit $\cV(\bF_q)$ of rank-$1$ points, which has size $q^2+q+1$, the orbit of rank-$3$ points, which has size $q^5-q^2$, and two orbits of rank-$2$ points. For $q$ even, the orbits of rank-$2$ points comprise the $q^2+q+1$ points of the nucleus plane $\mathcal{N}$, and the $(q^2-1)(q^2+q+1)$ points contained in conic planes but not in $\mathcal{N}\cup \cV(\bF_q) $. Therefore, the point-orbit distribution of a subspace $W$ of $\PG(5,q)$, $q$ even, is the $4$-tuple $[r_1,r_{2n},r_{2s},r_3]$, where $r_i$, $i\in \{1,3\}$, is the number of rank-$i$ points in $W$, $r_{2n}$ is the number of rank-$2$ points in $W \cap \mathcal{N}$, and $r_{2s}$ is the number of the remaining  rank-$2$ points in $W$. Similarly, we obtain the {\it line-}, {\it plane-}, {\it solid-}, and {\it hyperplane-orbit distributions} of $W$ for $r=1,2,3,4$ respectively. This data turns out to be very useful in determining the $K$-orbit of a subspace of $\PG(5,q)$. For example, if $q$ is even and $W$ is a solid of $\PG(5,q)$, then its line-orbit distribution completely determines its $K$-orbit \cite{solidsqeven}. In \cite{LaPoSh2021} it was shown that the line-orbit distributions are a complete invariant for the $K$-orbits of planes meeting the Veronese surface $\cV(\bF_q)$, $q$ odd.
The line orbits themselves were determined (for all $q$) in \cite{lines}, based on the canonical forms of tensors of format $(2,3,3)$ from \cite{canonical}.
\begin{Theorem}\label{lines} 
There are 15 $K$-orbits of lines in $\PG(5,q)$, $q$ even, as described in Table \ref{tableoflines}.
\end{Theorem}

Hyperplanes of $\PG(5,q)$ correspond to conics of $\PG(2,q)$ through the Veronese map $\nu$. In particular, we have four $K$-orbits of hyperplanes defined as follows: $\cH_{1}$, $\cH_{2r}$ and $\cH_{2i}$ denote the $K$-orbits of hyperplanes corresponding to the $\PGL(3,q)$-orbits of double lines, pairs of real lines and pairs of conjugate imaginary lines in $\PG(2,q)$, respectively, and $\cH_3$ denotes the $K$-orbit of hyperplanes corresponding to the $\PGL(3,q)$-orbit of non-singular conics in $\PG(2,q)$.

\FloatBarrier
\begin{table}[!htbp]
\begin{center}
\footnotesize{
\begin{tabular}[h]{ l l} 
 \hline
Orbits& Point-OD's\\ \hline
$o_5$&$[2,0,q-1,0]$\\
$o_6$& $[1,1,q-1,0]$\\
$o_{8,1}$ & $[1,0,1,q-1]$\\
$o_{8,2}$ & $[1,1,0,q-1]$\\
$o_9$ & $[1,0,0,q]$\\
$o_{10}$ & $[0,0,q+1,0]$\\
$o_{12,1}$ & $[0,q+1,0,0]$\\ 
$o_{12,2}$ & $[0,1,q,0]$\\ 
$o_{13,1}$ & $[0,1,1,q-1]$\\ 
$o_{13,2}$ & $[0,0,2,q-1]$\\ 
$o_{14}$ & $[0,0,3,q-2]$\\ 
$o_{15}$ & $[0,0,1,q]$\\ 
$o_{16,1}$ & $[0,1,0,q]$\\
$o_{16,2}$ & $[0,0,1,q]$\\  
$o_{17}$ & $[0,0,0,q+1]$\\ \hline

 \end{tabular}}
 \caption{\label{tableoflines}$K$-orbits of lines in $\PG(5,q)$, $q$ even.}
\end{center}
\end{table}

The following criterion for a conic  to be non-singular in $\PG(2,q)$, $q$ even, is well-known (see e.g. \cite{hirsch}).  

\begin{Lemma}\label{lem:non-singular}
The conic $\mathcal{C}=\cZ(h)$, where $h=a_{00}X^2+a_{01}XY+a_{02}XZ+a_{11}Y^2+a_{12}YZ+a_{22}Z^2\in \F2h[X,Y,Z]$ is non-singular if and only if $a_{00}a_{12}^2+a_{11}a_{02}^2+a_{22}a_{01}^2+a_{01}a_{02}a_{12}\neq 0$. 
\end{Lemma}

\subsection{Inflexion points of planes in $\PG(5,q)$}

Recall that with a plane $\pi$ in $\PG(5,q)$ we associated a planar cubic $\mathscr{C}(\pi)$.
As we will see later, studying the cubic curves associated with planes in $\PG(5,q)$ is a useful $K$-invariant to differentiate between non-equivalent planes. In this section we introduce some terminology which will be used in the sequel.

\begin{Remark}
Let us already point out that the study of these cubic curves is not sufficient to completely characterize each orbit. For instance, the representatives of the orbits $\Sigma_8$ and $\Sigma_9$ in Table \ref{tableplanes} share the same cubic curve $\mathcal{Z}(XZ^2)$. In this case, the fact that these two orbits are distinct can be seen by computing their intersection with the nucleus plane $\mathcal{N}$.
\end{Remark}

Let $\mathscr{C}$ be a cubic curve in $\PG(2,q)$ defined by $\mathscr{C}=\mathcal{Z}(f)$, where 
\begin{equation} \label{cubiccurve1}
\begin{split}
f(X,Y,Z)=a_{000}X^3+a_{011}XY^2+a_{022}XZ^2+a_{100}X^2Y+a_{111}Y^3\\
+a_{122}YZ^2+a_{200}X^2Z+a_{211}Y^2Z+a_{222}Z^3+a_{012}XYZ.
\end{split}
\end{equation}
 Then, $\mathscr{C}$ can be represented by $\mathscr{C}(A,a_{012})$ where  \begin{equation}\label{matrix}
A=\begin{bmatrix}a_{000}&a_{011}&a_{022}\\ a_{100}&a_{111}&a_{122}\\a_{200}&a_{211}&a_{222} \end{bmatrix}.\end{equation}

\begin{Definition}

The set of tangent lines to a cubic curve  $\mathscr{C}(A,a_{012})$ defined over a finite field of characteristic 2, known as the cubic envelope of $\mathscr{C}(A,a_{012})$, is the dual of the cubic curve defined by 
$${\mathscr{C}(\Phi(A),a_{012}^2)},$$
where 
\begin{equation}\Phi(A)= Adj(A)^T+a_{012}\begin{bmatrix}0&a_{022}&a_{011}\\ a_{122}&0&a_{100}\\a_{211}&a_{200}&0 \end{bmatrix}\end{equation} and $Adj(A)^T$ is the transpose of the adjoint matrix of $A$.

\end{Definition}

\begin{Definition} \label{defofinflexion}
An inflexion point of a cubic curve is a point of the curve whose tangent
meets the curve algebraically in a triple intersection.
\end{Definition}

The next lemma gives a characterisation of inflexion points of cubic curves defined over finite fields of characteristic $2$.

\begin{Lemma} \label{hessian} \cite[Theorem 3.5]{glynn}\\
Let  $\mathscr{C}(A,a_{012})$ be a cubic curve defined over a finite field of characteristic 2 with $a_{012}\neq0$. Then, the inflexion points of $\mathscr{C}(A,a_{012})$ are the non-singular points of $\mathscr{C}(A,a_{012})$ which lie on the cubic curve %$\mathscr{C}''$, where  $\mathscr{C}''=(\mathscr{C}')'
$\mathscr{C}(\Phi^2(A),a_{012}^4)$. The curve $\mathscr{C}(\Phi^2(A),a_{012}^4)$ is also known as the Hessian of $\mathscr{C}(A,a_{012})$.
\end{Lemma}

\begin{Remark}
 For fields of characteristic different from two, points of inflexion are defined as points of the intersection of the cubic with the classical Hessian (the determinant of the $3\times 3$ matrix of second derivatives), which is zero over characteristic two fields.
\end{Remark}

\begin{Definition}
We define inflexion points of a plane $\pi$ in $\PG(5,q)$ to be  inflexion points of its associated cubic curve in $\PG(2,q)$ defined as the determinant of the matrix representation of $\pi$, as explained above.
\end{Definition}

\section{The classification}

In this section we classify the $K$-orbits of planes which meet the Veronese surface $\cV(\Fq)$, $q=2^h$, in at least one point. The classification is divided into different parts, depending on the point-orbit distribution $[r_1,r_{2n},r_{2s},r_3]$. Figure \ref{fig:tree} gives an overview of the structure of the discussion, and can be interpreted as a decision tree for the classification, including geometric and combinatorial invariants for each of the orbits. The notation $\pi$ is used for a plane representing the $K$-orbit, $\mathscr{C}$ denotes the cubic curve associated with $\pi$, $\ell_i$ denotes a line of type $o_i$, $C$ denotes a non-singular conic, $T_6(C)$ denotes a tangent line to $C$ of type $o_6$, $\epsilon$ denotes the number of inflexion points of $\mathscr{C}$ and a representative $\pi$ has property $\Im$ if it meets a conic plane of $\cV(\Fq)$ in a line.

\begin{figure}[h]
\centering
\includegraphics[scale=0.3]{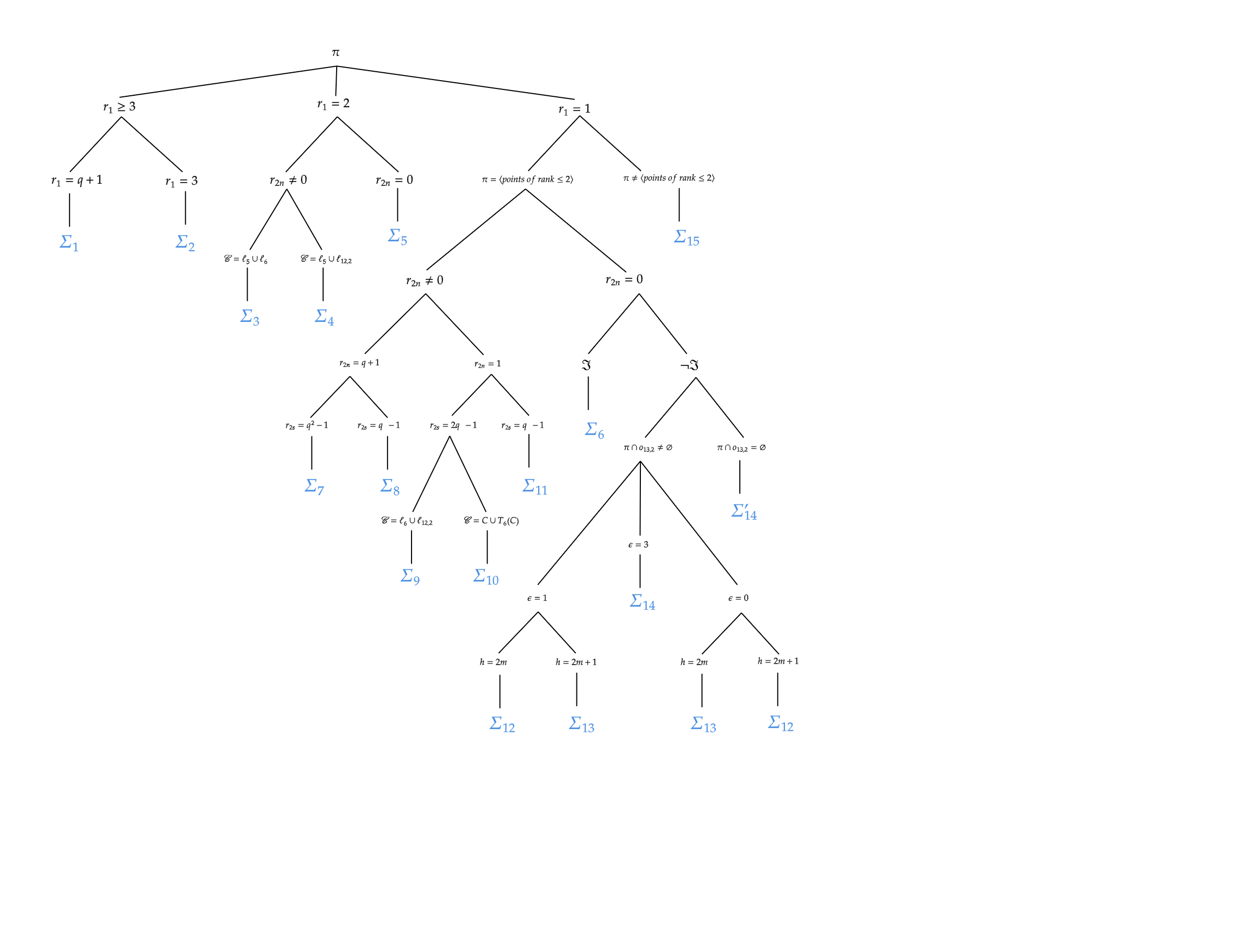}
\vspace{-5cm}\caption{Decision tree for $K$-orbits of planes meeting the Veronese variety in at least one point.}
\label{fig:tree}
\end{figure}

\subsection{Some non-existence results}

We start by proving non-existence results for planes with certain rank distributions. These results are obtained by giving bounds on the number of rank-2 points in planes of $\PG(5,q)$ meeting $\cV(\Fq)$ in one or two points.

\begin{Lemma}\label{aux}
There is no plane in $\PG(5,q)$ with rank distribution $[1,0,q^2+q]$.
\end{Lemma}
\begin{proof}
Let $Q_1$ be the unique rank-1 point in a plane $\pi\subset \PG(5,q)$ having no rank-2 points. By inspecting the point-orbit distributions of lines of $\PG(5,q)$ from Table \ref{tableoflines}, we conclude that all lines through $Q_1$ in $\pi$ must be of type $o_9$. Therefore, we may assume without loss of generality that $\pi=\langle Q_1,Q_2,Q_3\rangle$, where $\langle Q_1(1,0,0,0,0,0),Q_2(0,0,1,1,0,0)\rangle$ is the representative of the line orbit $o_9$ in \cite[Table 2]{lines} and $Q_3$ is a point of rank 3 with homogeneous coordinates $(0,a,0,b,c,d) \in \Fq^6$. As $Q_3$ has rank three, it follows that $a, d\neq 0$. Thus, we may take  $Q_3$ as the point $(0,1,0,a,b,c)$ for some $a,b,c \in \Fq$ with $c \neq 0$ and the representative of $\pi$ becomes %$\pi$ can be represented by
 $$\begin{bmatrix} x&y&z\\y&ay+z&by\\z&by&cy \end{bmatrix}.$$
The cubic curve $\mathscr{C}(\pi)$ then has the form $XF(Y,Z)+G(Y,Z)$, where $$F(Y,Z)=b^2Y^2+acY^2+cYZ,\:\: G(Y,Z)=aYZ^2+cY^3+Z^3.$$ 
Since $F$ defines a quadric on $\cZ(F)$ in $\PG(1,q)$ and each point of $\PG(1,q)$ not on $\cZ(F)$
 corresponds to a point in $\pi$ of rank $2$, it follows that $F$ must be identically zero. Therefore, $b=c=0$, a contradiction.

\end{proof}

\begin{Lemma}\label{2n}
There is no plane in $\PG(5,q)$ with rank distribution $[2,r_2<q,r_3]$.
\end{Lemma}
\begin{proof}
Let $Q_1, Q_2 \in \pi\cap \mathcal{V}(\Fq)$. Since points on the line $\langle Q_1,Q_2\rangle$ have rank at most $2$, it follows that $\pi$ has at least $q-1$ rank-$2$ points. Assume by way of contradiction that $r_2<q$. Then $r_2=q-1$, and all rank-$2$ points in $\pi$ lie on the line $\langle Q_1,Q_2\rangle$. Consequently, the cubic curve $\mathscr{C}$ defining points of rank at most $2$ in $\pi$ is the triple line $\langle Q_1,Q_2\rangle$. Assume, without loss of generality, that $\pi=\langle Q_1,Q_2,Q_3\rangle$ where $Q_1=\nu(e_1)$, $Q_2=\nu(e_2)$ and $Q_3$ is the point of rank 3 with corresponding matrix $$M_{Q_3}=\begin{bmatrix}
0&a&b\\a&0&c\\b&c&d
\end{bmatrix},$$ for some $a,b,c,d \in \Fq$. By the above it follows that the cubic curve $\mathscr{C}=\mathcal{Z}( dXYZ+c^2XZ^2+a^2dZ^3+b^2YZ^2) $ associated with $\pi$ is a triple line. But this implies that $c=d=0$, which is a contradiction since the rank of $M_{Q_3}$ is $3$.
\end{proof}

\subsection{Planes containing at least three rank-$1$ points}\label{atleast3}

Let $\pi$ be a plane in $\PG(5,q)$ with at least three rank-$1$ points. As $\cV(\Fq)$ is a cap, it follows that no three rank-$1$ points in $\pi$ are collinear. Thus, $\pi$ can be viewed as $\pi=\langle Q_1,Q_2,Q_3\rangle$ where $Q_i=\nu(q_i)$ for $1\leq i\leq 3$. We differentiate between the following two possibilities:\\\par

$(i)$ If $q_1,q_2$ and $q_3$ are collinear in $\PG(2,q)$, then $Q_1, Q_2, Q_3 \in\mathcal{C}( Q_1,Q_2)$, and $\pi$ is a conic plane of $\cV(\Fq)$. Since $\PGL(3,q)$ acts transitively on lines in $\PG(2,q)$, $K$ acts transitively on conic planes. We denote this orbit by $\Sigma_1$. By taking $\langle q_1,q_2\rangle$ as the line $\langle e_1,e_2\rangle$ we obtain the following representative
 $$\Sigma_1:\begin{bmatrix} x&y&.\\y&z&.\\.&.&.          \end{bmatrix}.$$ \par
 \begin{Lemma}
 The point-orbit distribution of a plane in $\Sigma_1$ is $[q+1,1,q^2-1,0]$.
 \end{Lemma}
 \begin{proof}
 Points of rank one in $\Sigma_1$ correspond to points on the quadric $\mathcal{Z}(XZ+Y^2)$. The remaining $q^2$ points in $\Sigma_1$ are of rank two, where only the point parametrized by $(x,y,z)=(0,1,0)$ is contained in the nucleus plane $\mathcal{N}$. Therefore, the point-orbit distribution of a plane in $\Sigma_1$ is $[q+1,1,q^2-1,0]$.
 \end{proof}

$(ii)$ If $q_1,q_2$ and $q_3$ are non-collinear in $\PG(2,q)$, then without loss of generality we may take $q_i=\langle e_i\rangle$ for $1\leq i\leq 3$. This gives a new plane orbit $\Sigma_2$ whose representative is 
 $$\Sigma_2:\begin{bmatrix} x&.&.\\.&y&.\\.&.&z          \end{bmatrix}.$$
 
  \begin{Lemma}
 The point-orbit distribution of a plane in $\Sigma_2$ is $[3,0,3q-3,q^2-2q+1]$ and $\Sigma_1\neq \Sigma_2$.
 \end{Lemma}
 \begin{proof}
 Points of rank at most two in $\Sigma_2$ correspond to points on the cubic curve $\mathscr{C}_2=\mathcal{Z}(XYZ)$, which is the union of three non-concurrent lines. The rank-1 points are the vertices of the triangle. The remaining $3q-3$ points on $\mathscr{C}_2$ correspond to rank-$2$ points in $\Sigma_2$, and none of these is contained in the nucleus plane $\mathcal{Z}(Y_0,Y_3,Y_5)$. Therefore, the point-orbit distribution of a plane in $\Sigma_2$ is $[3,0,3q-3,q^2-2q+1]$ and $\Sigma_1\neq\Sigma_2$.
  \end{proof}

 \subsection{Planes containing exactly two rank-$1$ points}\label{2rankone}
In this section we consider planes of $\PG(5,q)$ intersecting the Veronese surface in exactly two points. Let $\pi$ be such a plane containing the rank-1 points $Q_1$ and $Q_2$. By Lemma \ref{2n}, there exists a rank-$2$ point $Q_3$ in $\pi$ not lying on the line $\langle Q_1,Q_2\rangle$. Let $U=\mathcal{C}(Q_1,Q_2)\cap \mathcal{C}(Q_3)$ where $\mathcal{C}(Q_1,Q_2)$ and $\mathcal{C}(Q_3)$ are the two conics associated with $\{Q_1,Q_2\}$ and $Q_3$ as described in Section \ref{pre}. We study the cases where $U \in \{Q_1,Q_2\}$ and $U \not\in \{Q_1,Q_2\}$ separately.\\\par

\begin{figure}[h]
\centering
\tikzset{every picture/.style={line width=0.75pt}}

\begin{tikzpicture}[x=0.75pt,y=0.75pt,yscale=-1,xscale=1]

%Shape: Ellipse [id:dp11781605313433197] 
\draw   (77.23,111.3) .. controls (70.5,102.53) and (77.49,85.89) .. (92.83,74.13) .. controls (108.17,62.37) and (126.05,59.94) .. (132.77,68.7) .. controls (139.5,77.47) and (132.51,94.11) .. (117.17,105.87) .. controls (101.83,117.63) and (83.95,120.06) .. (77.23,111.3) -- cycle ;
%Shape: Ellipse [id:dp1383730468616362] 
\draw   (78.9,119.67) .. controls (88.16,113.64) and (104.21,121.89) .. (114.76,138.09) .. controls (125.31,154.29) and (126.35,172.3) .. (117.1,178.33) .. controls (107.84,184.36) and (91.79,176.11) .. (81.24,159.91) .. controls (70.69,143.71) and (69.65,125.7) .. (78.9,119.67) -- cycle ;
%Shape: Ellipse [id:dp8242269052884308] 
\draw  [fill={rgb, 255:red, 0; green, 0; blue, 0 }  ,fill opacity=1 ] (84.14,117.28) .. controls (84.14,116.18) and (85.31,115.28) .. (86.75,115.28) .. controls (88.2,115.28) and (89.37,116.18) .. (89.37,117.28) .. controls (89.37,118.39) and (88.2,119.28) .. (86.75,119.28) .. controls (85.31,119.28) and (84.14,118.39) .. (84.14,117.28) -- cycle ;
%Shape: Parallelogram [id:dp14397701253164508] 
\draw   (74.2,36) -- (189,36) -- (139.8,117) -- (25,117) -- cycle ;
%Shape: Parallelogram [id:dp9023634922301895] 
\draw   (137.47,116.88) -- (188.35,205.49) -- (75.89,205.6) -- (25,117) -- cycle ;
%Shape: Ellipse [id:dp5382604631833678] 
\draw  [fill={rgb, 255:red, 0; green, 0; blue, 0 }  ,fill opacity=1 ] (130.14,86.28) .. controls (130.14,85.18) and (131.31,84.28) .. (132.75,84.28) .. controls (134.2,84.28) and (135.37,85.18) .. (135.37,86.28) .. controls (135.37,87.39) and (134.2,88.28) .. (132.75,88.28) .. controls (131.31,88.28) and (130.14,87.39) .. (130.14,86.28) -- cycle ;
%Shape: Ellipse [id:dp012664617606693573] 
\draw   (260.23,111.3) .. controls (253.5,102.53) and (260.49,85.89) .. (275.83,74.13) .. controls (291.17,62.37) and (309.05,59.94) .. (315.77,68.7) .. controls (322.5,77.47) and (315.51,94.11) .. (300.17,105.87) .. controls (284.83,117.63) and (266.95,120.06) .. (260.23,111.3) -- cycle ;
%Shape: Ellipse [id:dp6887583614777473] 
\draw   (261.9,119.67) .. controls (271.16,113.64) and (287.21,121.89) .. (297.76,138.09) .. controls (308.31,154.29) and (309.35,172.3) .. (300.1,178.33) .. controls (290.84,184.36) and (274.79,176.11) .. (264.24,159.91) .. controls (253.69,143.71) and (252.65,125.7) .. (261.9,119.67) -- cycle ;
%Shape: Ellipse [id:dp5261100780934755] 
\draw  [fill={rgb, 255:red, 0; green, 0; blue, 0 }  ,fill opacity=1 ] (267.14,117.28) .. controls (267.14,116.18) and (268.31,115.28) .. (269.75,115.28) .. controls (271.2,115.28) and (272.37,116.18) .. (272.37,117.28) .. controls (272.37,118.39) and (271.2,119.28) .. (269.75,119.28) .. controls (268.31,119.28) and (267.14,118.39) .. (267.14,117.28) -- cycle ;
%Shape: Parallelogram [id:dp18063307961110397] 
\draw   (257.2,36) -- (372,36) -- (322.8,117) -- (208,117) -- cycle ;
%Shape: Parallelogram [id:dp7783043904573699] 
\draw   (320.47,116.88) -- (371.35,205.49) -- (258.89,205.6) -- (208,117) -- cycle ;
%Shape: Ellipse [id:dp34812911605244223] 
\draw  [fill={rgb, 255:red, 0; green, 0; blue, 0 }  ,fill opacity=1 ] (313.14,86.28) .. controls (313.14,85.18) and (314.31,84.28) .. (315.75,84.28) .. controls (317.2,84.28) and (318.37,85.18) .. (318.37,86.28) .. controls (318.37,87.39) and (317.2,88.28) .. (315.75,88.28) .. controls (314.31,88.28) and (313.14,87.39) .. (313.14,86.28) -- cycle ;
%Shape: Ellipse [id:dp7842300342800432] 
\draw  [fill={rgb, 255:red, 0; green, 0; blue, 0 }  ,fill opacity=1 ] (273.21,74.13) .. controls (273.21,73.02) and (274.39,72.13) .. (275.83,72.13) .. controls (277.28,72.13) and (278.45,73.02) .. (278.45,74.13) .. controls (278.45,75.23) and (277.28,76.13) .. (275.83,76.13) .. controls (274.39,76.13) and (273.21,75.23) .. (273.21,74.13) -- cycle ;

% Text Node
\draw (140,190.4) node [anchor=north west][inner sep=0.75pt]  [font=\scriptsize]  {$\langle \mathcal{C}( Q_{3}) \rangle $};
% Text Node
\draw (120.2,38.4) node [anchor=north west][inner sep=0.75pt]  [font=\scriptsize]  {$\langle \mathcal{C}( Q_{1} ,Q_{2}) \rangle $};
% Text Node
\draw (82,100.4) node [anchor=north west][inner sep=0.75pt]  [font=\scriptsize]  {$Q_{1}$};
% Text Node
\draw (136,77.4) node [anchor=north west][inner sep=0.75pt]  [font=\scriptsize]  {$Q_{2}$};
% Text Node
\draw (266,101.4) node [anchor=north west][inner sep=0.75pt]  [font=\scriptsize]  {$U$};
% Text Node
\draw (259,57.4) node [anchor=north west][inner sep=0.75pt]  [font=\scriptsize]  {$Q_{1}$};
% Text Node
\draw (319,77.4) node [anchor=north west][inner sep=0.75pt]  [font=\scriptsize]  {$Q_{2}$};
% Text Node
\draw (303.2,38.4) node [anchor=north west][inner sep=0.75pt]  [font=\scriptsize]  {$\langle \mathcal{C}( Q_{1} ,Q_{2}) \rangle $};
% Text Node
\draw (324,191.4) node [anchor=north west][inner sep=0.75pt]  [font=\scriptsize]  {$\langle \mathcal{C}( Q_{3}) \rangle $};

\end{tikzpicture}
\caption{Configurations associated with cases $(i)$ and $(ii)$, respectively.}
\end{figure}
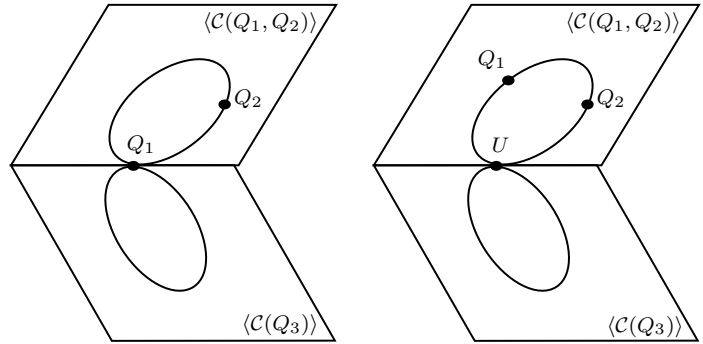

\boxed{$(i)$} If $U \in \{Q_1,Q_2\}$, then without loss of generality we may assume that $U=Q_1$. Let $q_1$, $q_2$ and $l_3$ be the preimages under $\nu$ of $Q_1$, $Q_2$ and $\mathcal{C}(Q_3)$ respectively. % Desargesian plane:
As the elation group $E(q_1,\langle q_1,q_2\rangle)$, with centre $q_1$ and axis $\langle q_1,q_2\rangle$, acts transitively on the affine points of $\PG(2,q)\setminus \langle q_1,q_2\rangle$, it follows that we may fix $\langle q_1,q_2\rangle$ and $l_3$ as $\langle e_1,e_2\rangle$ and $\langle e_1,e_3\rangle$. Hence the points $Q_1,Q_2$ can be represented by 
$$M_{Q_1}=\begin{bmatrix}1&0&0\\0&0&0\\0&0&0  \end{bmatrix}\:\:\:\: \text{and} \:\:\:\:M_{Q_2}=\begin{bmatrix}0&0&0\\0&1&0\\0&0&0  \end{bmatrix}. $$

Note that $Q_3$ must lie on the tangent $T_{Q_1}(\mathcal{C}(Q_3))$ since otherwise the plane $\pi$ would meet $\cV(\Fq)$ in at least three points, contradicting the assumption that $\pi$ contains exactly two points of rank one. Therefore, $\pi$ is completely determined by $\langle Q_1,Q_2\rangle$ and $T_{Q_1}(\mathcal{C}(Q_3))=\mathcal{Z}(Y_5)$, where $\mathcal{C}(Q_3)=\mathcal{Z}(Y_0Y_5+Y_2^2)\cap\mathcal{Z}(Y_1,Y_3,Y_4)$, leading to the $K$-orbit $\Sigma_3$ 
represented by
$$\Sigma_3:\begin{bmatrix} x&.&z\\.&y&.\\z&.&.        \end{bmatrix}.$$\par

\begin{Lemma}\label{sigma3}
The point-orbit distribution of a plane in $\Sigma_3$ is $[2,1,2q-2,q^2-q]$. In particular, $\Sigma_3 \not\in \{\Sigma_1, \Sigma_2\}$.
\end{Lemma}
\begin{proof}
Let $\pi_3$ be the above representative of $\Sigma_3$. Points of rank at most $2$ in $\pi_3$ correspond to points on the cubic curve $\mathscr{C}_3=\mathcal{Z}(YZ^2)$. Among these $2q+1$ points, there are exactly two rank-$1$ points corresponding to points of $\mathcal{Z}(Z,XY)$ and a unique rank-$2$ point in $\mathcal{N}\cap \pi_3=\mathcal{Z}(X,Y)$ with parametrized coordinates $(0,0,1)$. Therefore, the point-orbit distribution of a plane in $\Sigma_3$ is $[2,1,2q-2,q^2-q]$. In particular, $\Sigma_3 \not \in \{\Sigma_1,\Sigma_2\}$.
\end{proof}

\boxed{$(ii)$} If $U \not\in \{Q_1,Q_2\}$, then without loss of generality, let $q_1=\langle e_1\rangle$, $q_2=\langle e_2\rangle $ and $u=\langle e_1+e_2\rangle$.
By the transitivity properties of the group $\PGL(3,q)$, it follows that we may also fix $l_3=\nu^{-1}(\mathcal{C}(Q_3))$ as  $\langle e_1+e_2,e_3\rangle$. The following three possibilities for $Q_3$ in the conic plane $\langle\mathcal{C}(Q_3)\rangle $ may occur:

$(ii$-$a)$ $Q_3=N(\mathcal{C}(Q_3))$ , $(ii$-$b)$ $Q_3 \in T_{U}(\mathcal{C}(Q_3))\setminus\{N(\mathcal{C}(Q_3)),U\}$ or $(ii$-$c)$ $Q_3 \in \langle\mathcal{C}(Q_3)\rangle \setminus (\mathcal{C}(Q_3)\cup T_{U}(\mathcal{C}(Q_3)))$.\\\par

\boxed{$(ii$-$a)$} In this case 
$Q_3$ has coordinates $(0,0,1,0,1,0)$ and we obtain the orbit represented by
$$\Sigma_4:\begin{bmatrix} x&.&z\\.&y&z\\z&z&.        \end{bmatrix}.$$

\begin{Lemma}\label{sigma4}
The point-orbit distribution of a plane in $\Sigma_4$ is $[2,1,2q-2,q^2-q]$. In particular, $\Sigma_4 \not\in \{\Sigma_1,\Sigma_2,\Sigma_3\}$.
\end{Lemma}
\begin{proof}
Let $\pi_4$ be the above representative of $\Sigma_4$. The rank-$1$ points in $\pi_4$ are the two points with parametrized coordinates $(1,0,0)$ and $(0,1,0)$. The remaining points on the cubic curve $\mathscr{C}_4=\mathcal{Z}(Z^2(X+Y))$ correspond to points of rank 2, where only the point parametrized by $(0,0,1)$ lies in the nuleus plane. Therefore, the point-orbit distribution of a plane in $\Sigma_4$ is $[2,1,2q-2,q^2-q]$ and $\Sigma_4 \not\in \{\Sigma_1,\Sigma_2\}$ by their different point-orbit distributions. Finally, by observing that $\mathscr{C}_3$, the cubic curve associated with $\pi_3$, is the union of two lines of type $o_5$ and $o_6$, while $\mathscr{C}_4$ is the union of two lines of type $o_5$ and $o_{12,2}$, we can deduce that $\Sigma_3$ and $\Sigma_4$,  which share the same point-orbit distribution, are also distinct. 
\end{proof}

\par

\boxed{$(ii$-$b)$} If $Q_3 \in T_{U}(\mathcal{C}(Q_3))\setminus\{N(\mathcal{C}(Q_3)), U\}$, then $Q_3$ has coordinates $(a,a,1,a,1,0)$ for some $a \in \Fq\setminus\{0\}$. 
It follows that $\pi$, represented by $$\pi_a:\begin{bmatrix} x+az&az&z\\az&y+az&z\\z&z&.\end{bmatrix}$$ for some $a \in \Fq\setminus\{0\}$,  intersects the nucleus plane in a unique point $Q_3'$ with homogeneous coordinates $(0,a,1,0,1,0)$. By considering the two possibilities where $U'=\mathcal{C}(Q_1,Q_2)\cap \mathcal{C}_{Q_3'}$ belongs to $\{Q_1,Q_2\}$ or not, we end up with the plane $\pi=\langle Q_1,Q_2,Q_3'\rangle$ which belongs to one of the orbits $\Sigma_3$ or $\Sigma_4$. Hence, this case will not define a new orbit.
\\

\boxed{$(ii$-$c)$}\label{page55} Finally, if $Q_3 \in \langle\mathcal{C}(Q_3)\rangle \setminus( \mathcal{C}(Q_3)\cup T_{U}(\mathcal{C}(Q_3)))$, then let $R_3=\nu(r_3)=\langle U,Q_3\rangle \cap \mathcal{C}(Q_3)$. The subgroup in $\PGL(3,q)$ 
 stabilising $\{ u,q_1,q_2\}$ and $l_3$ contains the elation group with center $u$ and axis $\langle q_1,q_2\rangle$, and thus it acts transitively on points of $l_3\setminus\{u\}$. Hence, without loss of generality we may also fix $r_3$.  Then the point $Q_3$ can be written as $Q_3=\langle U,R_3\rangle \cap \langle N(\cC(Q_3)),S_3\rangle$, where $S_3$ is the unique point of $\cC(Q_3)$ which lies on the tangent of $\cC(Q_3)$ through $Q_3$.
 Since the group of homologies with centre $r_3$ and axis $\langle q_1,q_2\rangle$  acts transitively on points of $l_3\setminus\{u,r_3\}$, we may also fix $S_3$.
  This shows that any choice of $Q_3$ as a point on the secant $\langle U, R_3\rangle$ defines the same orbit. For $Q_3=(0,0,1,0,1,1)$ we obtain the orbit $\Sigma_5$ with representative  
$$\Sigma_5:\begin{bmatrix} x&.&z\\.&y&z\\z&z&z       \end{bmatrix}.$$

\begin{Lemma}
The point-orbit distribution of a plane in $\Sigma_5$ is $[2,0,2q-2,q^2-q+1]$. In particular, $\Sigma_5 \not \in \{\Sigma_1,\Sigma_2,\Sigma_3,\Sigma_4\}$.
\end{Lemma}
\begin{proof}

Let $\pi_5$ be the above representative of $\Sigma_5$. Points of rank at most $2$ in $\pi$ correspond to points of the cubic curve $\mathscr{C}_5=\mathcal{Z}(XYZ+XZ^2+YZ^2)$, which intersect the nucleus plane $\mathcal{N}$ trivially and the Veronese surface $\cV(\Fq)$ in exactly two points. Namely, points with parametrized coordinates $(1,0,0)$ and $(0,1,0)$.
Therefore, the point-orbit distribution of a plane in $\Sigma_5$ is $[2,0,2q-2,q^2-q+1]$ and $\Sigma_5 \not \in \{\Sigma_1,\Sigma_2,\Sigma_3,\Sigma_4\}$.
\end{proof}

\subsection{Planes containing one rank-1 point and spanned by points of rank at most 2}\label{qwe}

In this section we investigate planes of $\PG(5,q)$ which are spanned by points of rank at most $2$ and which meet the Veronese surface in exactly one point. 
Let $\pi=\langle Q_1,Q_2,Q_3\rangle$ be such a plane where $rank(Q_1)=1$ and $rank(Q_2)=rank(Q_3)=2$, and consider the two conics $\mathcal{C}(Q_2)$ and $\mathcal{C}(Q_3)$ associated with $Q_2$ and $Q_3$ respectively. Denote by $q_1$, $l_2$ and $l_3$ the respective preimages of $Q_1$, $\mathcal{C}(Q_2)$ and $\mathcal{C}(Q_3)$ under the Veronese embedding. By symmetry, we are left with the following possibilities: 
\begin{itemize}
\item [(a)] $l_2=l_3$,
\item [(b)] $q_1 = l_2\cap l_3$, 
\item [(c)] $q_1 \in l_2\setminus  l_3$, and,
\item [(d)] $q_1 \not\in l_2\cup l_3$. \end{itemize}

\begin{figure}[h]
\centering

\tikzset{every picture/.style={line width=0.75pt}}

\begin{tikzpicture}[x=0.75pt,y=0.75pt,yscale=-1,xscale=1]

%Straight Lines [id:da1547332451226089] 
\draw    (111.45,141.76) -- (87.02,70.64) ;
%Shape: Ellipse [id:dp9495721249194595] 
\draw  [color={rgb, 255:red, 0; green, 0; blue, 0 }  ,draw opacity=1 ][fill={rgb, 255:red, 0; green, 0; blue, 0 }  ,fill opacity=1 ] (109.15,140.07) .. controls (109.15,139.14) and (110.18,138.39) .. (111.45,138.39) .. controls (112.71,138.39) and (113.74,139.14) .. (113.74,140.07) .. controls (113.74,141) and (112.71,141.76) .. (111.45,141.76) .. controls (110.18,141.76) and (109.15,141) .. (109.15,140.07) -- cycle ;
%Straight Lines [id:da001853891485938508] 
\draw    (209.15,142.8) -- (184.72,71.69) ;
%Shape: Ellipse [id:dp5878562817168149] 
\draw  [color={rgb, 255:red, 0; green, 0; blue, 0 }  ,draw opacity=1 ][fill={rgb, 255:red, 0; green, 0; blue, 0 }  ,fill opacity=1 ] (223.43,90.92) .. controls (223.43,89.99) and (224.46,89.24) .. (225.72,89.24) .. controls (226.99,89.24) and (228.02,89.99) .. (228.02,90.92) .. controls (228.02,91.85) and (226.99,92.6) .. (225.72,92.6) .. controls (224.46,92.6) and (223.43,91.85) .. (223.43,90.92) -- cycle ;
%Straight Lines [id:da7470255556688197] 
\draw    (312.96,141.76) -- (288.53,70.64) ;
%Straight Lines [id:da8058534070822221] 
\draw    (312.96,141.76) -- (339.13,70.64) ;
%Shape: Ellipse [id:dp4068753837377488] 
\draw  [color={rgb, 255:red, 0; green, 0; blue, 0 }  ,draw opacity=1 ][fill={rgb, 255:red, 0; green, 0; blue, 0 }  ,fill opacity=1 ] (310.66,140.07) .. controls (310.66,139.14) and (311.69,138.39) .. (312.96,138.39) .. controls (314.22,138.39) and (315.25,139.14) .. (315.25,140.07) .. controls (315.25,141) and (314.22,141.76) .. (312.96,141.76) .. controls (311.69,141.76) and (310.66,141) .. (310.66,140.07) -- cycle ;
%Straight Lines [id:da3892807748871021] 
\draw    (436.83,143.85) -- (412.4,72.73) ;
%Straight Lines [id:da7780106898811812] 
\draw    (436.83,143.85) -- (463,72.73) ;
%Shape: Ellipse [id:dp6060967741744376] 
\draw  [color={rgb, 255:red, 0; green, 0; blue, 0 }  ,draw opacity=1 ][fill={rgb, 255:red, 0; green, 0; blue, 0 }  ,fill opacity=1 ] (434.53,142.16) .. controls (434.53,141.24) and (435.56,140.48) .. (436.83,140.48) .. controls (438.09,140.48) and (439.12,141.24) .. (439.12,142.16) .. controls (439.12,143.09) and (438.09,143.85) .. (436.83,143.85) .. controls (435.56,143.85) and (434.53,143.09) .. (434.53,142.16) -- cycle ;
%Shape: Brace [id:dp7067871107701311] 
\draw   (69.57,158.49) .. controls (69.6,163.16) and (71.95,165.48) .. (76.62,165.45) -- (146.41,165.02) .. controls (153.08,164.98) and (156.43,167.29) .. (156.46,171.96) .. controls (156.43,167.29) and (159.74,164.94) .. (166.41,164.9)(163.41,164.92) -- (236.21,164.47) .. controls (240.88,164.44) and (243.2,162.1) .. (243.17,157.43) ;
%Shape: Ellipse [id:dp663306578654592] 
\draw  [color={rgb, 255:red, 0; green, 0; blue, 0 }  ,draw opacity=1 ][fill={rgb, 255:red, 0; green, 0; blue, 0 }  ,fill opacity=1 ] (419.7,101.38) .. controls (419.7,100.45) and (420.73,99.7) .. (422,99.7) .. controls (423.27,99.7) and (424.29,100.45) .. (424.29,101.38) .. controls (424.29,102.31) and (423.27,103.06) .. (422,103.06) .. controls (420.73,103.06) and (419.7,102.31) .. (419.7,101.38) -- cycle ;
%Straight Lines [id:da5150666292816777] 
\draw    (556.34,142.8) -- (531.91,71.69) ;
%Straight Lines [id:da8648508321275774] 
\draw    (556.34,142.8) -- (582.51,71.69) ;
%Shape: Ellipse [id:dp13875383009284348] 
\draw  [color={rgb, 255:red, 0; green, 0; blue, 0 }  ,draw opacity=1 ][fill={rgb, 255:red, 0; green, 0; blue, 0 }  ,fill opacity=1 ] (554.04,141.12) .. controls (554.04,140.19) and (555.07,139.44) .. (556.34,139.44) .. controls (557.6,139.44) and (558.63,140.19) .. (558.63,141.12) .. controls (558.63,142.05) and (557.6,142.8) .. (556.34,142.8) .. controls (555.07,142.8) and (554.04,142.05) .. (554.04,141.12) -- cycle ;
%Shape: Ellipse [id:dp7649143393299256] 
\draw  [color={rgb, 255:red, 0; green, 0; blue, 0 }  ,draw opacity=1 ][fill={rgb, 255:red, 0; green, 0; blue, 0 }  ,fill opacity=1 ] (588.94,123.34) .. controls (588.94,122.41) and (589.96,121.66) .. (591.23,121.66) .. controls (592.5,121.66) and (593.53,122.41) .. (593.53,123.34) .. controls (593.53,124.27) and (592.5,125.02) .. (591.23,125.02) .. controls (589.96,125.02) and (588.94,124.27) .. (588.94,123.34) -- cycle ;

% Text Node
\draw (51.7,59.81) node [anchor=north west][inner sep=0.75pt]  [font=\scriptsize]  {$l_{2} =l_{3}$};
% Text Node
\draw (95.24,136.56) node [anchor=north west][inner sep=0.75pt]  [font=\scriptsize]  {$q_{1}$};
% Text Node
\draw (149.4,58.77) node [anchor=north west][inner sep=0.75pt]  [font=\scriptsize]  {$l_{2} =l_{3}$};
% Text Node
\draw (222.6,99.96) node [anchor=north west][inner sep=0.75pt]  [font=\scriptsize]  {$q_{1}$};
% Text Node
\draw (140.08,96.51) node [anchor=north west][inner sep=0.75pt]    {$or$};
% Text Node
\draw (145.94,173.52) node [anchor=north west][inner sep=0.75pt]   [align=left] {(a)};
% Text Node
\draw (425.96,173.52) node [anchor=north west][inner sep=0.75pt]   [align=left] {(c)};
% Text Node
\draw (274.68,58.77) node [anchor=north west][inner sep=0.75pt]  [font=\scriptsize]  {$l_{2}$};
% Text Node
\draw (342.72,59.81) node [anchor=north west][inner sep=0.75pt]  [font=\scriptsize]  {$l_{3}$};
% Text Node
\draw (296.75,136.56) node [anchor=north west][inner sep=0.75pt]  [font=\scriptsize]  {$q_{1}$};
% Text Node
\draw (301.21,172.47) node [anchor=north west][inner sep=0.75pt]   [align=left] {(b)};
% Text Node
\draw (398.55,60.86) node [anchor=north west][inner sep=0.75pt]  [font=\scriptsize]  {$l_{2}$};
% Text Node
\draw (466.59,61.9) node [anchor=north west][inner sep=0.75pt]  [font=\scriptsize]  {$l_{3}$};
% Text Node
\draw (407.53,94.73) node [anchor=north west][inner sep=0.75pt]  [font=\scriptsize]  {$q_{1}$};
% Text Node
\draw (420.68,135.47) node [anchor=north west][inner sep=0.75pt]  [font=\scriptsize]  {$u$};
% Text Node
\draw (545.47,172.47) node [anchor=north west][inner sep=0.75pt]   [align=left] {(d)};
% Text Node
\draw (518.06,59.81) node [anchor=north west][inner sep=0.75pt]  [font=\scriptsize]  {$l_{2}$};
% Text Node
\draw (586.1,60.86) node [anchor=north west][inner sep=0.75pt]  [font=\scriptsize]  {$l_{3}$};
% Text Node
\draw (540.19,134.43) node [anchor=north west][inner sep=0.75pt]  [font=\scriptsize]  {$u$};
% Text Node
\draw (595.96,117.74) node [anchor=north west][inner sep=0.75pt]  [font=\scriptsize]  {$q_{1}$};

\end{tikzpicture}

\caption{The configurations defined by cases (a), (b), (c) and (d) in Section \ref{qwe}.}
\end{figure}
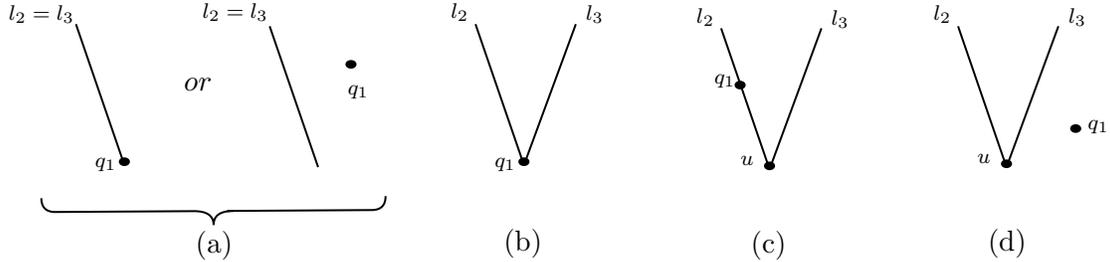

\subsubsection{(a) $l_2=l_3$} 

If $l_2=l_3$, then assume first that $q_1 \in l_2$.  In this case, $\pi$ becomes a conic plane and thus lies in $\Sigma_1$. Assume next that $q_1 \not\in   l_2$. As $\PGL(3,q)$ acts transitively on antiflags in $\PG(2,q)$ and  $\pi$ has a unique rank-1 point, it follows that we may fix $Q_1$ and $ \mathcal{C}(Q_2) $ as $\nu(\langle e_1\rangle)$ and $\nu(\langle e_2,e_3\rangle)$ respectively, where the line $\langle Q_2,Q_3\rangle$ must be external to $ \mathcal{C}(Q_2)$. Now, as the group stabilising $Q_1$ and $ \mathcal{C}(Q_2)$ acts transitively on external lines to $ \mathcal{C}(Q_2)$, we obtain a unique orbit of such planes which we label as $\Sigma_6$. Indeed, we may fix $Q_2Q_3$ as the line $\mathcal{Z}( Y_3+c Y_4+Y_5)$ where $Tr(c^{-1})=1$ 
to get the following representative $$\Sigma_6:\begin{bmatrix} x&.&.\\.&y+c z&z\\.&z&y        \end{bmatrix}.$$        \\\par

 \begin{Lemma}\label{sigma6}
The point-orbit distribution of a plane in $\Sigma_6$ is $[1,0,q+1,q^2-1]$. In particular, $\Sigma_6 \not\in \{\Sigma_1,\Sigma_2,\Sigma_3,\Sigma_4,\Sigma_5\}$.
\end{Lemma}
\begin{proof}
Let $\pi_6$ be the above representative of $\Sigma_6$. Points of rank at most $2$ in $\pi_6$ correspond to points on the cubic curve $\mathscr{C}_6=\mathcal{Z}(XY^2+cXYZ+XZ^2)$. In particular, points of rank one in $\pi_6$ correspond to points on $\mathcal{Z}(XY,XZ,Y^2+c YZ+Z^2)$. As $Tr(c^{-1})=1$, we obtain a unique rank-$1$ point parametrized by $(1,0,0)$. The remaining points on $\mathscr{C}_6$ parametrize $q+1$  rank-$2$ points in $\pi_6$, where none of these is contained in the nucleus plane $\mathcal{N}$. Therefore, the point-orbit distribution of a plane in $\Sigma_6$ is $[1,0,q+1,q^2-1]$, and thus $\Sigma_6 \not\in \{\Sigma_1,\Sigma_2,\Sigma_3,\Sigma_4,\Sigma_5\}$. 

\end{proof}

\begin{Lemma}
A plane $\pi \in \Sigma_6$ has $q+1$ lines in $o_{8,1}$ and a unique line in $o_{10}$.
\end{Lemma}
\begin{proof}
By Lemma \ref{sigma6}, $\pi$ intersects the nucleus plane trivially and has $q+1$ rank-$2$ points lying on the line $\langle Q_2,Q_3\rangle$. Therefore, $\pi$ has a unique line in $o_{10}$ and each of the $q+1$ lines through the rank-$1$ point $Q_1$ must have $q$ rank-3 points, and thus belongs to the line-orbit $o_{8,1}$.
\end{proof}

\subsubsection{(b) $q_1= l_2\cap l_3$}
We may fix $q_1$, $l_2$ and $l_3$ as $e_1$, $\langle e_1,e_2\rangle$ and $\langle e_1,e_3\rangle$ respectively. Furthermore, as $\pi$ contains a unique rank-$1$ point, it follows that $Q_2\in T_{Q_1}(\mathcal{C}(Q_2))$ and $Q_3\in T_{Q_1}(\mathcal{C}(Q_3))$. 
 Therefore, $\pi$ is completely determined by $Q_1$, $\mathcal{C}(Q_2)$ and $\mathcal{C}(Q_3)$. This yields to a unique $K$-orbit $\Sigma_7$  represented by
 $$\Sigma_7:\begin{bmatrix} x&y&z\\y&.&.\\z&.&. \end{bmatrix}.$$\par

\begin{Lemma}
The point-orbit distribution of a plane in $\Sigma_7$ is $[1,q+1,q^2-1,0]$. In particular, $\Sigma_7 \not\in  \{\Sigma_1,\Sigma_2,\Sigma_3,\Sigma_4,\Sigma_5,\Sigma_6\}$.
\end{Lemma}
\begin{proof}
It follows from the above representative that points of $\Sigma_7$ have rank at most two. Particularly, $\Sigma_7$ has a unique rank-$1$ point obtained for $y=z=0$ and $q+1$ points in the nucleus plane parametrized by $\{(0,y,z): (y,z)\in \PG(1,q)\}$. Therefore, the point-orbit distribution of a plane in $\Sigma_7$ is  $[1,q+1,q^2-1,0]$. Moreover, by comparing this property with the previous orbits, we conclude that $\Sigma_7 \not\in  \{\Sigma_1,\Sigma_2,\Sigma_3,\Sigma_4,\Sigma_5,\Sigma_6\}$.
\end{proof}

\subsubsection{(c) $q_1 \in l_2\setminus l_3$}\label{case c}
 If $q_1 \in l_2\setminus l_3$, then without loss of generality we may assume $U=\nu(u)=\mathcal{C}(Q_2)\cap \mathcal{C}(Q_3)$ and $Q_1=\nu(q_1)$ with $u=\langle e_2\rangle$ and $q_1=\langle e_1\rangle$. The elation group $E(u, l_2)$, with centre $u$ and axis $l_2=\langle u ,q_1\rangle$, acts transitively on the affine points of $\PG(2,q)\setminus l_2$, and thus we may also fix $l_3$ as $\langle e_2,e_3\rangle$. Since $\pi$ has a unique rank-$1$ point, it follows that $Q_2$ lies on the tangent line $T_{Q_1}(\mathcal{C}(Q_2))$, and hence $\pi=\langle T_{Q_1}(\mathcal{C}(Q_2)),Q_3\rangle$.

  We need to consider the following possibilities: 
  \begin{itemize}
    \item[$(c$-$i)$] $Q_3=N(\mathcal{C}(Q_3))$,
    \item[$(c$-$ii)$] $Q_3\in T_{U}(\mathcal{C}(Q_3))\setminus \{N(\mathcal{C}(Q_3)),U\}$, and
    \item[$(c$-$iii)$] $Q_3 \in \langle\mathcal{C}(Q_3)\rangle \setminus (\mathcal{C}(Q_3)\cup T_{U}(\mathcal{C}(Q_3)))$.
  \end{itemize}

\boxed{$(c$-$i)$} If $Q_3$ is the nucleus point $N(\mathcal{C}(Q_3))$, then the next lemma shows that $\pi$ defines a new orbit
$$\Sigma_8:\begin{bmatrix} x&y&.\\y&.&z\\.&z&.          \end{bmatrix}.$$

\begin{Lemma}\label{1}
The point-orbit distribution of a plane in $\Sigma_8$ is $[1,q+1,q-1,q^2-q]$. In particular, $\Sigma_8\not\in  \{\Sigma_1,\Sigma_2,\Sigma_3,\Sigma_4,\Sigma_5,\Sigma_6,\Sigma_7\}$.
\end{Lemma}
\begin{proof}
Points of rank at most $2$ in $\Sigma_8$ correspond to points on the cubic curve $\mathscr{C}_8=\mathcal{Z}(XZ^2)$. Among these $2q+1$ points, there is a unique rank-1 point lying on $\mathscr{C}_8\cap\mathcal{Z}(Y,Z)$ and $q+1$ points in the nucleus plane lying on $\mathscr{C}_8 \cap\mathcal{Z}(X)$. Therefore, the point-orbit distribution of a plane in $\Sigma_8$ is $[1,q+1,q-1,q^2-q]$ and $\Sigma_8\not\in  \{\Sigma_1,\Sigma_2,\Sigma_3,\Sigma_4,\Sigma_5,\Sigma_6,\Sigma_7\}$ by their distinct point-orbit distributions.
\end{proof}
\vspace{0.3 cm}

\boxed{$(c-ii)$} Assume now that $Q_3\in T_{U}(\mathcal{C}(Q_3))\setminus \{N(\mathcal{C}(Q_3)),U\}$. We claim that the plane $\pi$ is then completely determined by the quadruple $(q_1,u,s_1,s_2)$, where the images $S_1=\nu(s_1)$ and $S_2=\nu(s_2)$ are the points of intersection of a secant line of $\cC_3$ through $Q_3$, i.e. $Q_3=T_{U}(\mathcal{C}(Q_3))\cap \langle S_1,S_2\rangle$. Clearly this quadruple determines $Q_1$ and $Q_3$. Since $\pi=\langle T_{Q_1}(\mathcal{C}(Q_2)),Q_3\rangle$, and the pair $(U,Q_1)$ determines $T_{Q_1}(\mathcal{C}(Q_2))$, the claim follows. Since $\PGL(3,q)$ acts transitively on such quadruples (to see this, fix the collinear triple $(u,s_1,s_2)$, and note that the translation group with axis $l_3$ acts transitively on the affine plane) the case (c-ii) give rise to just one orbit $\Sigma_9$. Without loss of generality, we may choose $Q_3$ as $(0,0,0,1,1,0)$  
 to obtain the following representative $$\Sigma_{9}:\begin{bmatrix} x&y&.\\y&z&z\\.&z&.          \end{bmatrix}.$$\\\par

\begin{Lemma}\label{2}
The point-orbit distribution of a plane in $\Sigma_{9}$ is $[1,1,2q-1,q^2-q]$. In particular, $\Sigma_9\not\in  \{\Sigma_1,\Sigma_2,\Sigma_3,\Sigma_4,\Sigma_5,\Sigma_6,\Sigma_7,\Sigma_8\}$.
\end{Lemma}
\begin{proof}
Similar to case $\Sigma_8$, points of rank at most $2$ in $\Sigma_9$ correspond to points on the cubic curve $\mathscr{C}_9=\mathcal{Z}(XZ^2)$. In particular, $\Sigma_9$ has a unique rank-$1$ point parametrized by $(1,0,0)$ and a unique rank-$2$ point in $\mathcal{N}$ 
parametrized by $(0,1,0)$. Therefore, the point-orbit distribution of a plane in $\Sigma_{9}$ is $[1,1,2q-1,q^2-q]$, and thus $\Sigma_9\not\in  \{\Sigma_1,\Sigma_2,\Sigma_3,\Sigma_4,\Sigma_5,\Sigma_6,\Sigma_7,\Sigma_8\}$.
\end{proof}

\begin{Remark}
The two planes $\pi_8$ and $\pi_9$ define the same cubic curve $\mathcal{Z}(XZ^2)$, however they are not $K$-equivalent.
\end{Remark}

 \begin{figure}[h]
 \centering

\tikzset{every picture/.style={line width=0.75pt}} %set default line width to 0.75pt        

\begin{tikzpicture}[x=0.75pt,y=0.75pt,yscale=-1,xscale=1]

%Shape: Ellipse [id:dp11043951139563779] 
\draw   (101.28,138.64) .. controls (94.17,128.91) and (101.57,110.42) .. (117.8,97.36) .. controls (134.04,84.29) and (152.97,81.6) .. (160.08,91.33) .. controls (167.2,101.07) and (159.8,119.55) .. (143.56,132.62) .. controls (127.32,145.68) and (108.39,148.38) .. (101.28,138.64) -- cycle ;
%Shape: Ellipse [id:dp8306649790963767] 
\draw   (103.06,147.95) .. controls (112.86,141.25) and (129.85,150.41) .. (141.02,168.4) .. controls (152.18,186.4) and (153.29,206.41) .. (143.49,213.11) .. controls (133.69,219.81) and (116.7,210.65) .. (105.53,192.65) .. controls (94.37,174.66) and (93.26,154.65) .. (103.06,147.95) -- cycle ;
%Shape: Ellipse [id:dp5530861939566627] 
\draw  [fill={rgb, 255:red, 0; green, 0; blue, 0 }  ,fill opacity=1 ] (108.6,145.29) .. controls (108.6,144.07) and (109.84,143.07) .. (111.37,143.07) .. controls (112.9,143.07) and (114.14,144.07) .. (114.14,145.29) .. controls (114.14,146.52) and (112.9,147.52) .. (111.37,147.52) .. controls (109.84,147.52) and (108.6,146.52) .. (108.6,145.29) -- cycle ;
%Shape: Parallelogram [id:dp26608587670378214] 
\draw   (98.08,55) -- (219.59,55) -- (167.52,144.98) -- (46.01,144.98) -- cycle ;
%Shape: Parallelogram [id:dp7388305015240386] 
\draw   (165.12,145.03) -- (222,249) -- (102.89,248.95) -- (46.02,144.98) -- cycle ;
%Shape: Ellipse [id:dp835560938741567] 
\draw  [fill={rgb, 255:red, 0; green, 0; blue, 0 }  ,fill opacity=1 ] (125.82,79.05) .. controls (125.82,77.82) and (127.06,76.83) .. (128.59,76.83) .. controls (130.12,76.83) and (131.36,77.82) .. (131.36,79.05) .. controls (131.36,80.28) and (130.12,81.27) .. (128.59,81.27) .. controls (127.06,81.27) and (125.82,80.28) .. (125.82,79.05) -- cycle ;
%Shape: Ellipse [id:dp8649155366648293] 
\draw  [fill={rgb, 255:red, 0; green, 0; blue, 0 }  ,fill opacity=1 ] (106.56,198.16) .. controls (106.56,197.2) and (107.8,196.42) .. (109.33,196.42) .. controls (110.86,196.42) and (112.1,197.2) .. (112.1,198.16) .. controls (112.1,199.13) and (110.86,199.91) .. (109.33,199.91) .. controls (107.8,199.91) and (106.56,199.13) .. (106.56,198.16) -- cycle ;
%Shape: Ellipse [id:dp9788047107036848] 
\draw  [fill={rgb, 255:red, 0; green, 0; blue, 0 }  ,fill opacity=1 ] (104.63,223.35) .. controls (104.63,222.38) and (105.87,221.6) .. (107.4,221.6) .. controls (108.93,221.6) and (110.17,222.38) .. (110.17,223.35) .. controls (110.17,224.31) and (108.93,225.09) .. (107.4,225.09) .. controls (105.87,225.09) and (104.63,224.31) .. (104.63,223.35) -- cycle ;
%Straight Lines [id:da8451879022900424] 
\draw    (111.37,147.04) -- (107.4,236.43) ;
%Shape: Ellipse [id:dp8752236081099847] 
\draw  [fill={rgb, 255:red, 0; green, 0; blue, 0 }  ,fill opacity=1 ] (104.63,234.69) .. controls (104.63,233.72) and (105.87,232.94) .. (107.4,232.94) .. controls (108.93,232.94) and (110.17,233.72) .. (110.17,234.69) .. controls (110.17,235.65) and (108.93,236.43) .. (107.4,236.43) .. controls (105.87,236.43) and (104.63,235.65) .. (104.63,234.69) -- cycle ;
%Straight Lines [id:da7983466331805964] 
\draw [color={rgb, 255:red, 74; green, 144; blue, 226 }  ,draw opacity=1 ]   (143.49,213.11) -- (107.4,234.69) ;
%Straight Lines [id:da4042717567947063] 
\draw [color={rgb, 255:red, 74; green, 144; blue, 226 }  ,draw opacity=1 ]   (107.4,223.35) -- (96.81,171.88) ;
%Straight Lines [id:da4375574346588573] 
\draw    (172.15,92.09) -- (119.02,75.96) ;
%Shape: Ellipse [id:dp15179549878806675] 
\draw  [fill={rgb, 255:red, 0; green, 0; blue, 0 }  ,fill opacity=1 ] (149.99,86) .. controls (149.99,84.78) and (151.23,83.78) .. (152.76,83.78) .. controls (154.29,83.78) and (155.53,84.78) .. (155.53,86) .. controls (155.53,87.23) and (154.29,88.22) .. (152.76,88.22) .. controls (151.23,88.22) and (149.99,87.23) .. (149.99,86) -- cycle ;

% Text Node
\draw (180.03,237.64) node [anchor=north west][inner sep=0.75pt]  [font=\tiny]  {$\langle \mathcal{C}( Q_{3}) \rangle $};
% Text Node
\draw (173.97,58.51) node [anchor=north west][inner sep=0.75pt]  [font=\tiny]  {$\langle \mathcal{C}( Q_{2}) \rangle $};
% Text Node
\draw (146.12,72.78) node [anchor=north west][inner sep=0.75pt]  [font=\tiny]  {$Q_{1}$};
% Text Node
\draw (107.75,130.78) node [anchor=north west][inner sep=0.75pt]  [font=\tiny]  {$U$};
% Text Node
\draw (121.61,65.8) node [anchor=north west][inner sep=0.75pt]  [font=\tiny]  {$Q_{2}$};
% Text Node
\draw (88.7,214.96) node [anchor=north west][inner sep=0.75pt]  [font=\tiny]  {$Q_{3}$};
% Text Node
\draw (109.4,236.34) node [anchor=north west][inner sep=0.75pt]  [font=\tiny]  {$Q_{3} '$};
% Text Node
\draw (111.38,195.14) node [anchor=north west][inner sep=0.75pt]  [font=\tiny]  {$R_{3}$};

\end{tikzpicture}
 \caption{The configuration defined in case $(c$-$iii)$.}\label{fig:c-iii}
 \end{figure}
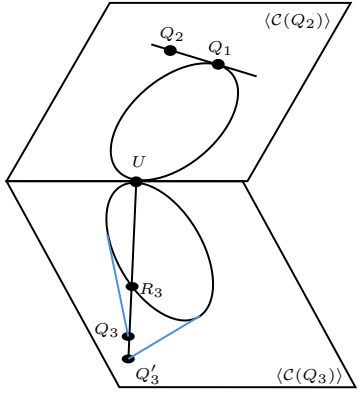

\boxed{$(c-iii)$} 
Finally, assume that $Q_3 \in \langle\mathcal{C}(Q_3)\rangle \setminus (\mathcal{C}(Q_3)\cup T_{U}(\mathcal{C}(Q_3)))$.  
Let $R_3=\nu(r_3)\neq U$ denote the intersection of $\cC(Q_3)$ with the secant line passing through $U$ and $Q_3$, and let $S_3=\nu(s_3)$ denote the point of intersection of the tangent line of $\cC(Q_3)$ through $Q_3$. Then $Q_3=\langle U,R_3\rangle \cap \langle N(\cC(Q_3)),S_3\rangle$.
This shows that in this case the plane $\pi$ is completely determined by the quadruple $(q_1,u,r_3,s_3)$. Since the group $\PGL(3,q)$ acts transitively on such quadruples, it follows that  (as illustrated in Figure \ref{fig:c-iii})
any other choice of a point $Q_3'\neq Q_3$ on $\langle\mathcal{C}(Q_3)\rangle \setminus( \mathcal{C}(Q_3)\cup T_{U}(\mathcal{C}(Q_3)))$ defines the same $K$-orbit $\Sigma_{10}$ represented by  
 $$\Sigma_{10}:\begin{bmatrix} x&y&.\\y&z&.\\.&.&z          \end{bmatrix},$$
 for the choice $Q_3=(0,0,0,1,0,1)$.

\begin{Lemma}\label{3}
The point-orbit distribution of a plane in $\Sigma_{10}$ is $[1,1,2q-1,q^2-q]$. In particular, $\Sigma_{10}\not\in  \{\Sigma_1,\Sigma_2,\Sigma_3,\Sigma_4,\Sigma_5,\Sigma_6,\Sigma_7,\Sigma_8,\Sigma_9\}$.
\end{Lemma}
\begin{proof}
Let $\pi_{10}$ be the above representative of $\Sigma_{10}$. Points of rank at most $2$ in $\pi_{10}$ correspond to the $2q+1$ points on the cubic curve $\mathscr{C}_{10}=\mathcal{Z}(XZ^2+Y^2Z)$. In particular, $\pi_{10}$ has a unique rank-$1$ point  parametrized by $(1,0,0)$ and a unique point lying on $\pi_{10}\cap\mathcal{N}=\mathcal{Z}(X,Z)$ parametrized by $(0,1,0)$.

Therefore, the point-orbit distribution of a plane in $\Sigma_{10}$ is $[1,1,2q-1,q^2-q]$ and $\Sigma_{10}\not\in  \{\Sigma_1,\Sigma_2,\Sigma_3,\Sigma_4,\Sigma_5,\Sigma_6,\Sigma_7,\Sigma_8\}$  by comparison to their point-orbit distributions. 
It remains to show that $\Sigma_9\neq \Sigma_{10}$. But this follows immediately by observing that $\mathscr{C}_{9}$ is the union of two lines of type $o_6$ and $o_{12,2}$, while $\mathscr{C}_{10}$ is the union of a non-singular conic and one of its tangent lines (which is a line of type $o_6$).

\end{proof}

\subsubsection{(d) $q_1 \not\in l_2\cup l_3$}
 Finally, assume that $q_1 \not\in l_2\cup l_3$. As before, let $U=\nu(u)=\mathcal{C}(Q_2)\cap \mathcal{C}(Q_3)$. We distinguish the following two configurations: 
 $(d$-$i)$ $\pi\cap \mathcal{N}\neq \emptyset$ and $(d$-$ii)$ $\pi\cap \mathcal{N}= \emptyset$, where $\mathcal{N}$ is the nucleus plane.

 \subsubsubsection{$(d$-$i)$ $\pi\cap \mathcal{N}\neq \emptyset$}

Since $\pi$ intersects the nucleus plane, we may relabel our points such that $Q_2\in \mathcal N$.

If $Q_3=N(\mathcal{C}(Q_3))$ then the plane $\pi$ meets the nucleus plane in a line $\langle Q_2,Q_3\rangle$. Since the set of $q+1$ nuclei of the $q+1$ conics on $\cV(\bF_q)$ through the point $Q_1$ on $\cV(\bF_q)$, are the points of a line $L$ in the nucleus plane (by Lemma \ref{lem:nuclei_pencil}) it follows that $\pi$ is the span of $Q_1$, $Q_2$, and the point $Q'_3=\langle Q_2,Q_3\rangle\cap L$. This implies that $\pi$ belongs to the orbit $\Sigma_8$.

Assume next that $Q_3 \in T_{U}(\mathcal{C}(Q_3))\setminus \{N(\mathcal{C}(Q_3)),U\}$.
Consider a point $R_2=\nu(r_2)$ on $\cC(Q_2)$ different from $U$, and note that this point determines $Q_2$ as the intersection of the tangent line of $\cC(Q_2)$ at $U$ with the tangent line of $\cC(Q_2)$ at $R_2$. Also, let $R_3=\nu(r_3)$, where $r_3$ is the intersection point of the lines $\langle r_2,q_1\rangle$ with $l_3$. The points on $T_{U}(\mathcal{C}(Q_3))\setminus \{N(\mathcal{C}(Q_3)),U\}$ are in one-to-one correspondence with the points on the line $l_3$, different from $u$ and $r_3$, since any such point $r_3'$ on $l_3$ defines a secant line $\langle R_3,\nu(r'_3)\rangle$ which meets the tangent line $T_{U}(\mathcal{C}(Q_3))$.
The homology group with center $u$ and axis $\langle q_1,r_2\rangle$ acts transitively on these points of $l_3$, it follows that any choice of $Q_3$ on $T_{U}(\mathcal{C}(Q_3))\setminus \{N(\mathcal{C}(Q_3)),U\}$ will lead to a plane $\langle Q_1,Q_2,Q_3\rangle$ in $\PG(5,q)$ which belongs to the same $K$-orbit.
Without loss of generality, we may fix $u$, $q_1$ and $r_2$ as $\langle e_1\rangle$, $\langle e_2+e_3\rangle$, $\langle e_2\rangle$, and $l_3$ as $\langle e_1,e_3\rangle$. By the above, we may also choose $Q_3$ on $T_{U}(\mathcal{C}(Q_3))\setminus \{N(\mathcal{C}(Q_3)),U\}$. For $Q_3(1,0,1,0,0,0)$ we obtain the plane $\pi$ represented by 
$$\begin{bmatrix} x&y&x\\y&z&z\\x&z&z          \end{bmatrix}. 
$$  
However, since $\pi$ intersects the Veronese surface in two points, namely $(1,0,0,0,0,0)$ and $(1,0,1,1,1,0)$, implying that $\pi\in\{\Sigma_3,\Sigma_4,\Sigma_5\}$, this case does not lead to a new $K$-orbit. 

It remains to consider the case where $Q_3\in \langle \mathcal{C}(Q_3)\rangle \setminus T_{U}(\mathcal{C}(Q_3))$.
Let $R_3=\nu(r_3)$ denote the unique point of intersection different from $U$ of the conic $\cC(Q_3)$ with the secant line $\langle U,Q_3\rangle$. 
This situation is depicted in Figure \ref{fig:Sigma_11}. In what follows we show that the planes
$\pi=\langle Q_1,Q_2,Q_3\rangle$ and $\pi'=\langle Q_1,Q_2,Q'_3\rangle$ belong to the same $K$-orbit.
In order to do so, denote by $R'_3=\nu(r'_3)$ the point of tangency of the  unique tangent line of $\cC(Q_3)$ through $ Q_3$. The plane $\pi$ is then uniquely determined by the quadruple $(q_1,l_2,r_3,r'_3)$ since $Q_2$ is the nucleus of $\nu(l_2)$, and $Q_3=\langle U, R_3\rangle \cap T_{U}(\nu(l_3))$, 
where $l_3=\langle r_3,r'_3\rangle$, $u=l_2\cap l_3$, and $U=\nu(u)$.
The group $\PGL(2,q)$ acts transitively on such quadruples, since the stabiliser of the triple $(q_1,r_3,r'_3)$ acts transitively on the points of $l_3$ different from $r_3$ and $r'_3$ (which implies we may choose $u$ and the homology group with center $q_1$ and axis $l_3$ acts transitively on the lines through $u$ (which implies we may choose $l_2$)).
Choosing coordinates as in the previous case, but now with $Q_3$ equal to the point with coordinates $(1,0,0,0,0,1)$, we obtain the orbit $\Sigma_{11}$ represented by 
$$\Sigma_{11}:\begin{bmatrix} x&y&.\\y&z&z\\.&z&x+z         \end{bmatrix} .$$

 \begin{figure}[h]
 \centering

\tikzset{every picture/.style={line width=0.75pt}} %set default line width to 0.75pt        

\begin{tikzpicture}[x=0.75pt,y=0.75pt,yscale=-1,xscale=1]

%Shape: Ellipse [id:dp09267537845832785] 
\draw   (121.28,158.64) .. controls (114.17,148.91) and (121.57,130.42) .. (137.8,117.36) .. controls (154.04,104.29) and (172.97,101.6) .. (180.08,111.33) .. controls (187.2,121.07) and (179.8,139.55) .. (163.56,152.62) .. controls (147.32,165.68) and (128.39,168.38) .. (121.28,158.64) -- cycle ;
%Shape: Ellipse [id:dp2184576039666295] 
\draw   (123.06,167.95) .. controls (132.86,161.25) and (149.85,170.41) .. (161.02,188.4) .. controls (172.18,206.4) and (173.29,226.41) .. (163.49,233.11) .. controls (153.69,239.81) and (136.7,230.65) .. (125.53,212.65) .. controls (114.37,194.66) and (113.26,174.65) .. (123.06,167.95) -- cycle ;
%Shape: Ellipse [id:dp4608329933479509] 
\draw  [fill={rgb, 255:red, 0; green, 0; blue, 0 }  ,fill opacity=1 ] (128.6,165.29) .. controls (128.6,164.07) and (129.84,163.07) .. (131.37,163.07) .. controls (132.9,163.07) and (134.14,164.07) .. (134.14,165.29) .. controls (134.14,166.52) and (132.9,167.52) .. (131.37,167.52) .. controls (129.84,167.52) and (128.6,166.52) .. (128.6,165.29) -- cycle ;
%Shape: Parallelogram [id:dp3849258574761163] 
\draw   (118.08,75) -- (239.59,75) -- (187.52,164.98) -- (66.01,164.98) -- cycle ;
%Shape: Parallelogram [id:dp029572256319862378] 
\draw   (185.12,165.03) -- (242,269) -- (122.89,268.95) -- (66.02,164.98) -- cycle ;
%Shape: Ellipse [id:dp02626733651011759] 
\draw  [fill={rgb, 255:red, 0; green, 0; blue, 0 }  ,fill opacity=1 ] (84.23,148) .. controls (84.23,146.77) and (85.47,145.78) .. (87,145.78) .. controls (88.53,145.78) and (89.77,146.77) .. (89.77,148) .. controls (89.77,149.23) and (88.53,150.22) .. (87,150.22) .. controls (85.47,150.22) and (84.23,149.23) .. (84.23,148) -- cycle ;
%Shape: Ellipse [id:dp5240803393559863] 
\draw  [fill={rgb, 255:red, 0; green, 0; blue, 0 }  ,fill opacity=1 ] (126.56,218.16) .. controls (126.56,217.2) and (127.8,216.42) .. (129.33,216.42) .. controls (130.86,216.42) and (132.1,217.2) .. (132.1,218.16) .. controls (132.1,219.13) and (130.86,219.91) .. (129.33,219.91) .. controls (127.8,219.91) and (126.56,219.13) .. (126.56,218.16) -- cycle ;
%Shape: Ellipse [id:dp7645060136669537] 
\draw  [fill={rgb, 255:red, 0; green, 0; blue, 0 }  ,fill opacity=1 ] (124.63,243.35) .. controls (124.63,242.38) and (125.87,241.6) .. (127.4,241.6) .. controls (128.93,241.6) and (130.17,242.38) .. (130.17,243.35) .. controls (130.17,244.31) and (128.93,245.09) .. (127.4,245.09) .. controls (125.87,245.09) and (124.63,244.31) .. (124.63,243.35) -- cycle ;
%Straight Lines [id:da6458554122553912] 
\draw    (131.37,167.04) -- (127.4,256.43) ;
%Shape: Ellipse [id:dp03611855947985099] 
\draw  [fill={rgb, 255:red, 0; green, 0; blue, 0 }  ,fill opacity=1 ] (124.63,254.69) .. controls (124.63,253.72) and (125.87,252.94) .. (127.4,252.94) .. controls (128.93,252.94) and (130.17,253.72) .. (130.17,254.69) .. controls (130.17,255.65) and (128.93,256.43) .. (127.4,256.43) .. controls (125.87,256.43) and (124.63,255.65) .. (124.63,254.69) -- cycle ;
%Straight Lines [id:da10298212906588056] 
\draw [color={rgb, 255:red, 74; green, 144; blue, 226 }  ,draw opacity=1 ]   (163.49,233.11) -- (127.4,254.69) ;
%Straight Lines [id:da29862539426658485] 
\draw [color={rgb, 255:red, 74; green, 144; blue, 226 }  ,draw opacity=1 ]   (127.4,243.35) -- (116.81,191.88) ;
%Straight Lines [id:da7919782438171898] 
\draw    (137.8,117.36) -- (125.63,124.7) -- (87,148) ;
%Shape: Ellipse [id:dp1833359178623133] 
\draw  [fill={rgb, 255:red, 0; green, 0; blue, 0 }  ,fill opacity=1 ] (135.04,117.36) .. controls (135.04,116.13) and (136.27,115.14) .. (137.8,115.14) .. controls (139.33,115.14) and (140.57,116.13) .. (140.57,117.36) .. controls (140.57,118.59) and (139.33,119.58) .. (137.8,119.58) .. controls (136.27,119.58) and (135.04,118.59) .. (135.04,117.36) -- cycle ;
%Straight Lines [id:da2673965727577583] 
\draw    (131.37,165.29) -- (87,148) ;
%Shape: Ellipse [id:dp7069160047371781] 
\draw  [fill={rgb, 255:red, 0; green, 0; blue, 0 }  ,fill opacity=1 ] (213.04,163.36) .. controls (213.04,162.13) and (214.27,161.14) .. (215.8,161.14) .. controls (217.33,161.14) and (218.57,162.13) .. (218.57,163.36) .. controls (218.57,164.59) and (217.33,165.58) .. (215.8,165.58) .. controls (214.27,165.58) and (213.04,164.59) .. (213.04,163.36) -- cycle ;

% Text Node
\draw (200.03,257.64) node [anchor=north west][inner sep=0.75pt]  [font=\tiny]  {$\langle \mathcal{C}( Q_{3}) \rangle $};
% Text Node
\draw (193.97,78.51) node [anchor=north west][inner sep=0.75pt]  [font=\tiny]  {$\langle \mathcal{C}( Q_{2}) \rangle $};
% Text Node
\draw (211.12,150.78) node [anchor=north west][inner sep=0.75pt]  [font=\tiny]  {$Q_{1}$};
% Text Node
\draw (127.75,150.78) node [anchor=north west][inner sep=0.75pt]  [font=\tiny]  {$U$};
% Text Node
\draw (78.61,151.8) node [anchor=north west][inner sep=0.75pt]  [font=\tiny]  {$Q_{2}$};
% Text Node
\draw (108.7,234.96) node [anchor=north west][inner sep=0.75pt]  [font=\tiny]  {$Q_{3}$};
% Text Node
\draw (129.4,256.34) node [anchor=north west][inner sep=0.75pt]  [font=\tiny]  {$Q_{3} '$};
% Text Node
\draw (131.38,215.14) node [anchor=north west][inner sep=0.75pt]  [font=\tiny]  {$R_{3}$};
% Text Node
\draw (125.12,104.78) node [anchor=north west][inner sep=0.75pt]  [font=\tiny]  {$R_{2}$};

\end{tikzpicture}
 \caption{The configuration defining $\Sigma_{11}$.}\label{fig:Sigma_11}
 \end{figure}
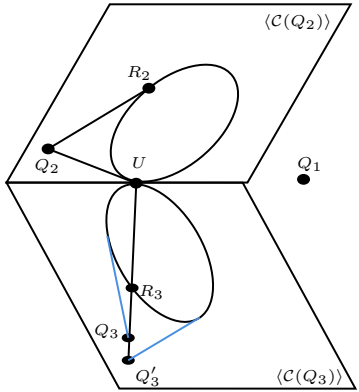

For future reference, we include the following. 
 
\begin{Lemma}
The line $\langle Q_1,Q_2\rangle$ in the representative of the orbit $\Sigma_{11}$ is of type $o_{8,2}$.
\end{Lemma}
\begin{proof}
This can be verified computationally, but it also follows from the following geometric argument. The line joining $Q_2$ and the unique rank-1 point $Q_1$ is either of type $o_6$ or $o_{8,2}$ by Table \ref{tableoflines}. Since lines of type $o_6$ in $\PG(5,q)$ are tangent lines to conics in $\cV(\Fq)$, and we are in the case where $Q_1$ does not belong to $\cC(Q_2)$, it follows that $\langle Q_1,Q_2\rangle$ is necessarily of type $o_{8,2}$. 
\end{proof}

\begin{Lemma}
The point-orbit distribution of a plane in $\Sigma_{11}$ is $[1,1,q-1,q^2]$. In particular, $\Sigma_{11}\not\in  \{\Sigma_1,\Sigma_2,\Sigma_3,\Sigma_4,\Sigma_5,\Sigma_6,\Sigma_7,\Sigma_8,\Sigma_9,\Sigma_{10}\}$.
\end{Lemma}
\begin{proof}
Let $\pi_{11}$ be the above representative of $\Sigma_{11}$. Points of rank at most $2$ in $\pi_{11}$ correspond to points on the cubic curve $\mathscr{C}_{11}=\mathcal{Z}(X^2Z+XY^2+Y^2Z)$. Particularly, $\pi_{11}$ has a unique rank-$1$ point and a unique rank-$2$ point lying on $\mathcal{N}\cap \pi_{11}=\mathcal{Z}(X,Z)$ with parameterized coordinates $(0,0,1)$ and $(0,1,0)$ respectively. 
 Therefore, the point-orbit distribution of a plane in $\Sigma_{11}$ is $[1,1,q-1,q^2]$ and $\Sigma_{11}$ is distinct from the previously defined orbits by comparison to their point-orbit distributions.
\end{proof}

 \subsubsubsection{$(d$-$ii)$ $\pi\cap \mathcal{N}= \emptyset$}\label{casedii}
Assume now that $\pi$ intersects the nucleus plane trivially, where the unique rank-$1$ point $Q_1$ is not lying on $\mathcal{C}(Q_2)\cup \mathcal{C}(Q_3)$.  
We begin with an essential lemma that gives a correspondence between types of lines  
 spanned by two rank-$2$ points in $\PG(5,q)$
 and their associated configurations defined by $\{\mathcal{C}(Q_2), \mathcal{C}(Q_3),U\}$, where $\mathcal{C}(Q_2)\neq\mathcal{C}(Q_3)$ and $U=\mathcal{C}(Q_2)\cap\mathcal{C}(Q_3)$.

\begin{Lemma}\label{linesspannedbyrank2}
Let $L$ be a line in $\PG(5,q)$ intersecting the nucleus plane trivially and spanned by two rank-$2$ points $R$ and $S$, where $\mathcal{C}(R)\neq\mathcal{C}(S)$. Then, $L\in \{o_{13,2},o_{14}\}$. Furthermore, if $V=\mathcal{C}(R)\cap\mathcal{C}(S)$, then $L \in o_{13,2}$ if and only if $R \in T_{V}(\mathcal{C}(R))$ and $S \not\in T_{V}(\mathcal{C}(S))$, and $L \in o_{14}$ if and only if $R \not\in T_{V}(\mathcal{C}(R))$ and $S \not\in T_{V}(\mathcal{C}(S))$. In particular, if $L \in o_{14}$, then the preimage  under the Veronese embedding of the three conics associated with the rank-2 points on $L$ are three non-concurrent lines in $\PG(2,q)$.
\end{Lemma}

\begin{proof}
The hyperplane spanned by $\mathcal{C}(R)$ and $\mathcal{C}(S)$ intersects $\mathcal{V}(\Fq)$ in $\mathcal{C}(R)\cup \mathcal{C}(S)$, and thus $L$ has no rank-$1$ points. It follows that $L\in\{o_{10},o_{13,2},o_{14}\}$ by Table \ref{tableoflines}. Since a line of type $o_{10}$ lies in a conic plane and $\mathcal{C}(R)\neq\mathcal{C}(S)$, we conclude that $L\in \{o_{13,2},o_{14}\}$. Let $L_{13,2}$ and $L_{14}$ be the representatives of $o_{13,2}$ and $o_{14}$ in \cite[Table 2]{lines}. The line $L_{13,2}$ has two rank-$2$ points with homogeneous coordinates $(0,1,0,1,0,0)$, and $(0,0,0,1,0,1)$, while the line $L_{14}$ has three rank-2 points defined as $$P_1(1,0,0,1,0,0),P_2(0,0,0,1,0,1),\mbox{~and~}P_3(1,0,0,0,0,1)\}.$$ 
By a direct computation, we have 
$$(0,1,0,1,0,0)\in T_{V}(\mathcal{C}((0,1,0,1,0,0)))$$ 
and 
$$ (0,0,0,1,0,1)\not\in T_{V}(\mathcal{C}((0,0,0,1,0,1)))$$ 
where $V=\nu(e_2)$. A similar computation shows that the three conics associated with $P_i$, $1\leq i\leq 3$; 
\begin{center}
$\mathcal{C}(P_1)=\mathcal{Z}(Y_0Y_3+Y_1^2,Y_2,Y_4,Y_5),$\\

$\mathcal{C}(P_2)=\mathcal{Z}(Y_3Y_5+Y_4^2,Y_0,Y_1,Y_2),$\\

$\mathcal{C}(P_3)=\mathcal{Z}(Y_0Y_5+Y_2^2,Y_1,Y_3,Y_4);$
\end{center}
intersect pairwise in $V\in\{U_{12}(0,0,0,1,0,0),U_{13}(1,0,0,0,0,0),U_{23}(0,0,0,0,0,1)\}$, where each pair $(P_i,P_j)$, $i< j$, has both of its points not lying on the tangent of their conics through $U_{ij}$.
\end{proof}

\begin{Remark}\label{P121314}
By inspecting point-orbit distributions of lines in $\PG(5,q)$, we can see that $\langle Q_1,Q_i\rangle\in o_{8,1} $ for $i=2,3$. Moreover, Lemma \ref{linesspannedbyrank2} implies that $\langle Q_2,Q_3\rangle\in \{o_{13,2},o_{14}\} $, where:  
 $\langle Q_2,Q_3\rangle\in o_{13,2}$ if and only if  $Q_2\in T_{U}(\mathcal{C}(Q_2))$ and $Q_3\not\in T_{U}(\mathcal{C}(Q_3))$, and  
 $\langle Q_2,Q_3\rangle\in o_{14}$ if and only if  $Q_2\not\in T_{U}(\mathcal{C}(Q_2))$ and $Q_3 \not\in T_{U}(\mathcal{C}(Q_3))$. 

\end{Remark}

Next, we consider the two possibilities where $\pi$ can be represented by $\langle Q_1,Q_2,Q_3\rangle$ where the unique rank-$1$ point $Q_1$ is not lying on $\mathcal{C}(Q_2)\cup \mathcal{C}(Q_3)$ such that: $(d$-$ii$-$A)$ $\langle Q_2,Q_3\rangle\in o_{13,2}$ or $(d$-$ii$-$B)$ $\langle Q_2,Q_3\rangle\not\in o_{13,2}$, i.e, $\pi$ has no line of type $o_{13,2}$ and $\langle Q_2,Q_3\rangle\in o_{14}$ by Lemma \ref{linesspannedbyrank2}.\\

\boxed{$(d$-$ii$-$A)$}
Let $\pi=\langle Q_1,Q_2,Q_3\rangle$ where the unique rank-$1$ point $Q_1$ is not lying on $\mathcal{C}(Q_2)\cup \mathcal{C}(Q_3)$, $\pi\cap \mathcal{N}=\emptyset$ and $\langle Q_2,Q_3\rangle\in o_{13,2}$. Without loss of generality, take $\langle Q_2,Q_3\rangle$ as the representative of $o_{13,2}$ in \cite[Table 2]{lines}, and let $Q_1$ be a point with homogeneous coordinates $\nu(a,b,c)$. As $L_i=\langle Q_1,Q_i\rangle\in o_{8,1} $ for $i=2,3$, it follows that $L_i$ 
 has a unique rank-1 point and a unique rank-2 point not contained in the nucleus plane, and thus $a,c\neq0$. Hence we may assume $Q_1=\nu(1,b,c)$. 
  Then $\pi$ can be represented by 
  $$
  \pi_{b,c}= \begin{bmatrix} x&bx+y&cx\\bx+y&b^2x+y+z&bcx\\ cx&bcx&c^2x+z \end{bmatrix}.
  $$
  This plane is $K$-equivalent to 
  $$
  \pi_{c} = \begin{bmatrix} x&y&cx\\y&y+z&.\\ cx&.&c^2x+z \end{bmatrix},
  $$ 
  which can be seen by defining $X$ as 
  $$X=\begin{bmatrix} 1&0&0\\b&1&0\\0&0&1\end{bmatrix}
  $$ 
  and observing that $X\pi_{c}X^T=\pi_{b,c}$.\\

Before proceeding with the study of planes of the form $\pi_c$, $c\neq0$, recall the definition and the characterisation of inflexion points in Definition \ref{defofinflexion} and Lemma \ref{hessian}. Note that, for fields of characteristic different from two, inflexion points are points of the intersection of the cubic with the Hessian (the determinant of the $3\times 3$ matrix of second derivatives), which becomes zero over fields of characteristic 2. For further details about inflexion points over fields of characteristic two, we refer to \cite{glynn}.

\begin{Lemma}\label{pic}
A plane $\pi_c$ with $c\neq 0$ has 
\begin{itemize}
\item three inflexion points if and only if $q\neq 4$, $Tr(c)=Tr(1)$ and $c^{-1}$ is admissible.
\item one inflexion point if and only if  $Tr(c)\neq Tr(1)$.
\item no inflexion points if and only if  $Tr(c)=Tr(1)$ and $c^{-1}$ is not admissible.
\end{itemize}
\end{Lemma}

\begin{proof}
Let $\mathscr{C}=\mathcal{Z}(f)$ be the cubic curve associated with $\pi_c$ defined by 
$$f=X(Z^2+YZ+c^2Y^2)+Y^2Z.$$ 
By Lemma \ref{hessian}, inflexion points of $\mathscr{C}$ correspond to non-singular points of $\mathscr{C}\cap\mathscr{C}''$, where $\mathscr{C}''=\mathcal{Z}(f'')$ and $$f''=X(Z^2+YZ+c^2Y^2)+Z^3+(1+c^2)Y^2Z+c^2Y^3.$$ 
The points in $\mathscr{C}\cap\mathscr{C}''$ therefore satisfy the equation:

 \begin{equation}\label{number} Z^3+c^2Y^2Z+c^2Y^3=0.\end{equation} 
The affine points $(X,1,Z)$ in $\pi_c\setminus \mathcal{Z}(Y)$ satisfy
 \begin{equation}Z^3+c^2Z+c^2=0.\end{equation} Let $\theta=c^{-1} Z$, then inflexion points of $\mathscr{C}$ correspond to solutions of \begin{equation}\label{theta} \theta^3+\theta +c^{-1}=0,\end{equation} where \eqref{theta} has three solutions if and only if $q\neq4$, $Tr(c)=Tr(1)$ and $c^{-1}$ is admissible, a unique solution if and only if  $Tr(c)\neq Tr(1)$, and no solution if and only if  $Tr(c)=Tr(1)$ and $c^{-1}$ is not admissible (see Lemma \ref{cubicequation}).
\end{proof}

\begin{Lemma}\label{inflexion}
Let $q=2^h$, $h>1$. There exist $c_{0}$ and $c_{1}$ in $\Fq\setminus\{0\}$ such that $\pi_{c_0}$ has no inflexion points and $\pi_{c_1}$ has a unique inflexion point. Moreover, if $h>2$, then there exists $c_3\in \Fq\setminus\{0\}$ such that $\pi_{c_3}$ has three inflexion points.
\end{Lemma}

\begin{proof}
This is a consequence of having exactly $\lfloor \frac{q-2}{6} \rfloor$ admissible scalars in $\Fq\setminus\{0\}$, $q\neq 4$ \cite[Lemma 1]{berlekamp}, and by noting that $Tr$ is a $\frac{q}{2}$-to-$1$ map.
\end{proof}
\begin{Remark}\label{inflexionrmk}
Let $q=2^h$, $h>1$. Lemma \ref{inflexion} implies the existence of at least three $K$-orbits of planes of the form $\pi_c$, when $h>2$, and at least two $K$-orbits of planes of the form $\pi_c$, when $h=2$. In particular, we denote by
\begin{itemize}
\item $\Sigma_{12}$ the union of $K$-orbits of planes represented by $\pi_c$ where  $Tr(c)=1$ and $c^{-1}$ is not admissible if $h$ is odd.
\item $\Sigma_{13}$ the union of $K$-orbits of planes represented by $\pi_c$ where $Tr(c)=0$ and $c^{-1}$ is not admissible if $h$ is even.
\item $\Sigma_{14}$ the union of $K$-orbits of planes represented by $\pi_c$ where $h>2$, $Tr(c)=Tr(1)$ and $c^{-1}$ is admissible.
\end{itemize}
\end{Remark}

\begin{Lemma}\label{phi14}
For $q=2^h>4$, inflexion points of planes in $\Sigma_{14}$ are collinear.  Furthermore, there exists a one-to-one correspondence between planes in $\Sigma_{14}$ and lines in the $K$-orbit $o_{14}$.
\end{Lemma}
\begin{proof}
Consider $\pi_c$ as the plane defined in Section \ref{casedii}, where $c$ is an admissible scalar in $\Fq\setminus\{0\}$. By Lemma 1 in \cite{berlekamp}, inflexion points of $\pi_c$ are the points parametrized by $(\frac{z_i}{z_i^2+z_i+c^2},1,z_i)$ where $$z_1=(1+v+v^{-1})^2,\:  z_2=\frac{v(1+v+v^{-1})^2}{v+v^{-1}},\;z_3=\frac{v^{-1}(1+v+v^{-1})^2}{v+v^{-1}},$$ and $v \in \Fq \setminus\mathbb{F}_4$. In particular, these points are collinear lying on the line $L_v$ with parametrized dual coordinates  $$[(v+v^{-1})(1+v+v^{-1})^2, (v+v^{-1})(1+v+v^{-1})^2, \frac{(v+v^{-1})^2+(v+v^{-1})^4+(v^3+v^{-3})^2}{(v+v^{-1})^4}].$$ We call $L_v$ an {\it inflexion line}. By Lemma \ref{linesspannedbyrank2} the line $L_v$ is of type $o_{14}$. First we prove that no two distinct planes in $\Sigma_{14}$ have the same inflexion line $L$. Without loss of generality, we may start by fixing $L$ as the representative of $o_{14}$ in \cite[Table 2]{lines}. More precisely, let $E_1=(1,0,0,1,0,0)$, $E_2=(0,0,0,1,0,1)$, and $E_3=(1,0,0,0,0,1)$ be the three inflexion points on $L$ parametrised by $(0,1,0)$, $(0,0,1)$ and $(0,1,1)$ respectively, and consider $Q_{a,b,c}=\nu(a,b,c)$ as a point on $\mathcal{V}(\Fq)$. If $\pi_{a,b,c}=\langle L,Q_{a,b,c}\rangle$ is a plane of type $\Sigma_{14}$, then $\langle Q_{a,b,c}, E_i\rangle\in o_{8,1}$, $1\leq i\leq 3$. This implies that $a,b,c \neq 0$. Therefore, we may assume without loss of generality that $a=1$, $Q_{a,b,c}=Q_{b,c}$ and $\pi_{a,b,c}=\pi_{b,c}$, where $$\pi_{b,c}=\begin{bmatrix}   x+y& bx&cx\\bx&b^2x+y+z&bcx\\cx&bcx&c^2x+z\end{bmatrix}.$$ The cubic curve $\mathscr{C}_{b,c}$ associated with $\pi_{b,c}$ is defined by 
\begin{equation}
XZ^2+c^2XY^2+Y^2Z+YZ^2+(1+b^2+c^2)XYZ=0.\end{equation}
If $1+b+c=0$, then $\pi_{b,c}$ intersects the nucleus plane $\mathcal{N}$ 
 in a unique point parametrised by $(1,1,1+b^2)$, a contradiction as planes in $\Sigma_{14}$ have no intersection with the nucleus plane. Therefore, we may assume that $1+b+c\neq0$. By Lemma \ref{hessian}, inflexion points of $\pi_{b,c}$ are non-singular points of $\mathscr{C}_{b,c}\cap\mathscr{C}_{b,c}''$, where $\mathscr{C}_{b,c}''=\mathcal{Z}(h_{b,c})$, $\alpha=(1+b^2+c^2)$ and \begin{equation}\begin{split}h_{b,c}= &c^2\alpha^5XY^2+\alpha^5XZ^2+c^2(1+b^2)\alpha Y^3+\alpha((1+b^2)+\alpha^3(b^2+c^2))YZ^2\\ & +\alpha(c^2(b^2+c^2)+\alpha^3(1+b^2))Y^2Z+(b^2+c^2)\alpha Z^3.\end{split}\end{equation}
 Imposing the conditions: $E_i \in \mathscr{C}_{b,c}''$, $1\leq i\leq 3$, implies that 
\begin{align}
c^2(1+b^2)\alpha+\alpha((1+b^2)+\alpha^3(b^2+c^2))+\alpha(c^2(b^2+c^2)+\alpha^3(1+b^2))+(b^2+c^2)\alpha&=0,
\end{align}
and \begin{align}c^2(1+b^2)\alpha&=(b^2+c^2)\alpha=0\end{align}.
  As $\alpha,c\neq0$, we get $b=c=1$. 
Therefore, every line in $o_{14}$ is the inflexion line of a unique plane in $\Sigma_{14}$, and thus we obtain a one-to-one correspondence between the set of planes in $\Sigma_{14}$ and the set of lines in $o_{14}$ being their inflexion lines.
\end{proof}

For future reference we extract the one-to-one correspondence between lines of type $o_{14}$ and planes of type $\Sigma_{14}$ which was established in the proof of Lemma \ref{phi14}.
\begin{Corollary}\label{cor:bijection}
A line $L$ of type $o_{14}$ in $\PG(5,q)$, $q>4$ and $q$ even, is contained in a unique plane (of type $\Sigma_{14}$) which meets the Veronese variety $\cV(\bF_q)$ in exactly one point, and which intersects the secant variety of $\cV(\bF_q)$ in a cubic curve with three $\bF_q$-rational inflexion points which lie on $L$.
\end{Corollary}
\begin{proof}
See proof of Lemma \ref{phi14}.
\end{proof}

\begin{Lemma}\label{sigma14unique}
For $q=2^h>4$, the planes in $\Sigma_{14}$ form one $K$-orbit.
\end{Lemma}
\begin{proof}
 Consider the plane  $\pi_{1,1}$ defined in Lemma \ref{phi14}. The stabiliser of $\pi_{1,1}$ in $K$, denoted by $K_{\pi_{1,1}}$, is the intersection of the two subgroups of $K$  stabilising the unique rank-1 point $Q$ and the inflexion line $L$, i.e,  
  $K_{\pi_{1,1}}=K_Q\cap K_L$. By \cite{lines}, we have $K_L\cong \operatorname{Sym}_3$, being the group represented by the six $3\times 3$ permutation matrices. Moreover, a matrix $g=(g_{ij})\in \GL(3,q)$ stabilises $Q$ if and only if $g_{11}+g_{12}+g_{13}=g_{21}+g_{22}+g_{23}=g_{31}+g_{32}+g_{33}.$ Therefore, $K_Q\cong E_q^2\rtimes\GL(2,q)$, and thus $K_{\pi_{1,1}}\cong \operatorname{Sym}_3$. Additionally, since the set $\Sigma_{14}$ has $|K|/6$ planes by Lemma \ref{phi14}, it follows that $\Sigma_{14}$ is equal to the $K$-orbit of $\pi_{1,1}$ in $\PG(5,q)$. Therefore, planes in $\Sigma_{14}$ define a unique $K$-orbit represented by $\pi_{1,1}$.
\end{proof}
By the above, we can represent the orbit $\Sigma_{14}$ by $\pi_{1,1}$
$$\Sigma_{14}:\begin{bmatrix}   x+y& x&x\\x&x+y+z&x\\x&x&x+z\end{bmatrix}.$$

The next lemma implies that we have two orbits
  $$
  \Sigma_{12}: \begin{bmatrix} x&y&cx\\y&y+z&.\\ cx&.&c^2x+z \end{bmatrix},
  $$
  with $Tr(c)=1$ and $c^{-1}$ is not admissible if $h$ is odd,
  and 
$$
\Sigma_{13}: \begin{bmatrix} x&y&cx\\y&y+z&.\\ cx&.&c^2x+z \end{bmatrix},
  $$
with $Tr(c)=0$ and $c^{-1}$ is not admissible if $h$ is even.

\begin{Remark}
In what follows, the notations $o_i(\bF_{q^j})$ and $\Sigma_i(\bF_{q^j})$ are used to denote orbits of lines and planes considered over extension fields $\bF_{q^j}$ of $\bF_q$. Furthermore, for a line  $L$ and a plane $\pi$ in $\PG(5,q)$, we denote by $L(\bF_{q^s})$ and $\pi(\bF_{q^s})$ their extensions over $\bF_{q^s}$, $s\in\{2,3\}$.
\end{Remark}

\begin{Lemma}\label{q>4even}
For $q>2$, the planes in $\Sigma_{12}$ form one $K$-orbit, and planes in $\Sigma_{13}$ form one $K$-orbit. 
\end{Lemma}
\begin{proof}
 
Recall that $\Sigma_{12}(\Fq)$ consists of all planes $\pi_{c_1}$, where $c_1\neq 0$, $Tr(c_1)=1$, and $\Sigma_{13}(\Fq)$ consists of all planes $\pi_{c_0}$, where $Tr(c_0)=0$ ($c_0\neq 0$). If $h$ is even, then $\pi_{c_1}$ has a unique inflexion point while $\pi_{c_0}$ has none. If $h$ is odd, then it is the other way around. Since the proof is essentially the same in both cases, we only give it for $h$ even.

By expanding $\pi_{c_1}$ to $\mathbb{F}_{q^2}$ and $\pi_{c_0}$ to $\mathbb{F}_{q^3}$, we obtain two planes $\pi_{c_1}(\mathbb{F}_{q^2})\subset\PG(5,q^2)$ and  $\pi_{c_0}(\mathbb{F}_{q^3})\subset \PG(5,q^3)$ belonging to the orbit $\Sigma_{14}(\bF_{q^s})$, $s\in\{2,3\}$. By Corollary \ref{cor:bijection}, each plane is uniquely determined by an inflexion line belonging to the orbit $o_{14}(\bF_{q^s})$, $s\in\{2,3\}$, say $L_1(\mathbb{F}_{q^2})$ and $L_0(\mathbb{F}_{q^3})$.
Let $\sigma_1$ (resp. $\sigma_0$) be the collineation of $\PG(5,q^2)$ (resp. $\PG(5,q^3)$) induced by the Frobenius automorphism $x\mapsto x^q$ of $\mathbb{F}_{q^2}$ (resp. $\mathbb{F}_{q^3}$).  
Since $\pi_{c_1}(\mathbb{F}_{q^2})$ has a unique $\Fq$-rational inflexion point and two $\mathbb{F}_{q^2}$-conjugate inflexion points, and $\pi_{c_0}(\mathbb{F}_{q^3})$ has three $\mathbb{F}_{q^3}$-conjugate inflexion points, it follows that $L_1=\pi_{c_1}\cap L_1(\mathbb{F}_{q^2})\in o_{15}(\bF_{q})$ and $L_0=\pi_{c_0}\cap L_0(\mathbb{F}_{q^3})\in o_{17}(\bF_{q})$. Note that, $L_1$ cannot belong to the orbit $o_{16,2}(\bF_{q})$ since the representative of this orbit in \cite[Table 2]{lines} is spanned by $(0,0,1,1,0,0)$ and $(0,0,0,0,1,1)$, which, over $\mathbb{F}_{q^2}$, generate a line belonging to $o_{16,2}(\bF_{q^2})$. Therefore, $L_1$ and $L_0$ are uniquely determined in $\pi_{c_1}$ and $\pi_{c_0}$, respectively. Moreover, by Corollary \ref{cor:bijection}, these lines uniquely determine the planes $\pi_{c_1}$ and $\pi_{c_0}$, since their respective extensions define a unique inflexion line in $o_{14}(\bF_{q^s})$, $s\in\{2,3\}$. 
Hence, there exists a one-to-one correspondence between planes in $\Sigma_{12}$ (resp. $\Sigma_{13}$) and lines in $o_{15}$ (resp. $o_{17}$). 
Using the stabiliser of these lines as computed in \cite{lines}, we thus obtain  $|K|/2$ planes in $\Sigma_{12}$ and $|K|/3$ planes in $\Sigma_{13}$. Let $K_{\pi_{c_i}}$ and $K_{L_{i}}$ be the stabilisers in $K$ of $\pi_{c_i}$ and $L_{i}$ respectively, $i\in\{0,1\}$. 

In order to show that the planes in $\Sigma_{12}$ form one $K$-orbit, it suffices to show that $K_{\pi_{c_1}}=K_{L_{12}}$. Since an element of $K$ fixing a plane must fix the cubic defined by it, clearly $K_{\pi_{c_1}}\leq K_{L_{12}}$. For the converse containment, consider an element $g\in K_{L_{12}}$. As before, we denote the unique point of rank $1$ contained in $\pi_{c_1}$ by $Q$. It suffices to show that $Q^g=Q$. Both planes, $\pi(\bF_{q^2})$ and $\pi^g(\bF_{q^2})$ are planes of $\PG(5,q^2)$ which belongs to the orbit $\Sigma_{14}(\bF_{q^2})$. The plane $\pi(\bF_{q^2})$
defines a cubic whose inflexion points lie on the line $L_{12}(\bF_{q^2})$, and the plane $\pi^g(\bF_{q^2})$ defines a cubic whose inflexion points lie on the line
$L_{12}^g(\bF_{q^2})=L_{12}(\bF_{q^2})$. By Corollary \ref{cor:bijection}, the plane $\pi^g(\bF_{q^2})=\pi(\bF_{q^2})$, and in particular $Q^g=Q$. This implies that the planes in $\Sigma_{12}$ form a single $K$-orbit.

Since the proof for $\Sigma_{13}$ is completely analogous, we do not include it here.
\end{proof}

We extract the following interesting facts from the proof of the previous lemma.

\begin{Corollary}\label{bijection_12_13}
(i) A line $L$ of type $o_{15}$ in $\PG(5,q)$, $q>2$ and $q$ even,  is contained in a unique plane which meets the Veronese variety $\cV(\bF_q)$ in exactly one point, and which intersects the secant variety $\cV(\bF_{q^2})$ in a cubic with three $\bF_{q^2}$-rational inflexion points which lie on $L(\bF_{q^2})$.

(ii) A line $L$ of type $o_{17}$ in $\PG(5,q)$, $q>2$ and $q$ even,  is contained in a unique plane which meets the Veronese variety $\cV(\bF_q)$ in exactly one point, and which intersects the secant variety $\cV(\bF_{q^3})$ in a cubic with three $\bF_{q^3}$-rational inflexion points which lie on $L(\bF_{q^3})$.
\end{Corollary}
\begin{proof}
The proof follows by the correspondence between  $\{o_{15},o_{17}\}$ and  $\{\Sigma_{12},\Sigma_{13}\}$ defined in Lemma  \ref{q>4even}, and by the geometric-combinatorial characteristics of planes in $\Sigma_{12}$ and $\Sigma_{13}$ and their extensions over $\bF_{q^s}$, $s\in\{2,3\}$, mentioned in Lemmas \ref{pic} and \ref{q>4even}.
\end{proof}

\begin{Lemma}
The point-orbit distributions of planes in $\Sigma_{12}$, $\Sigma_{13}$ and $\Sigma_{14}$ are given by $[1,0,q+1,q^2-1]$, $[1,0,q-1,q^2+1]$ and $[1,0,q\mp 1,q^2\pm1]$, respectively. In particular, these orbits are distinct from each other and from the previously defined orbits $\Sigma_i$, $1\leq i\leq 11$.
\end{Lemma}
\begin{proof}
Consider the cubic curve associated with $\Sigma_{i}$, $i \in \{12,13,14\}$,  defined by \begin{equation}X(Z^2+YZ+c^2Y^2)+Y^2Z=0.\end{equation} 
If $Tr(c)=1$ (resp. $Tr(c)=0$), then $\Sigma_{12}$ (resp. $\Sigma_{13}$) has $q+1$ (resp. $q$) rank-2 points parameterized by \begin{equation}(\frac{y^2z}{z^2+yz+c^2y^2},y,z).\end{equation} 
Furthermore, depending on the $Tr(c)$, being $0$ if $h$ is even or $1$ if $h$ is odd, the orbit $\Sigma_{14}$ has either $q-1$ or $q+1$ rank-2 points. Therefore, point-orbit distributions of planes in $\Sigma_{12}$, $\Sigma_{13}$ and $\Sigma_{14}$ are given by $[1,0,q+1,q^2-1]$, $[1,0,q-1,q^2+1]$ and $[1,0,q\mp 1,q^2\pm1]$, respectively. 
Note that, the fact that the orbits $\{\Sigma_{12}, \Sigma_{13}, \Sigma_{14}\}$ are distinct from each other follows from comparing  their inflexion points (see Lemma \ref{inflexion} and Remark \ref{inflexionrmk}). The fact that these orbits are different from the previously defined orbits (with the exception of $\Sigma_6$ for $\Sigma_{12}$ and $\Sigma_{14}$) follows from comparing their point-orbit distributions. Finally, $\Sigma_6\not\in\{\Sigma_{12}, \Sigma_{14}\}$ since a plane in $\Sigma_6$ contains distinct points which define the same conic plane, while all rank-2 points of a plane in $\Sigma_{12}\cup \Sigma_{14}$ define distinct conic planes.
\end{proof}

\boxed{$(d$-$ii$-$B)$} 
Finally, assume that $\pi=\langle Q_1,Q_2,Q_3\rangle$, where $\pi\cap \mathcal{N}=\emptyset$, $Q_1$ is not lying on $\mathcal{C}(Q_2)\cup \mathcal{C}(Q_3)$ and $Q_2,Q_3$ are both not lying on the tangent of their conics through $U=\mathcal{C}(Q_2)\cap\mathcal{C}(Q_3)$. Provided that $q>2$, we prove the existence of such planes if and only if $q=4$. 

 Without loss of generality, let $q_1=\langle e_1\rangle$, $u=\langle e_3\rangle $, $l_2=\langle e_2,e_3\rangle$ and $l_3=\langle e_3,e_1+e_2\rangle$, where $Q_1=\nu(q_1)$, $U=\nu(u)$, $\mathcal{C}(Q_2)=\nu(l_2)$ and $\mathcal{C}(Q_3)=\nu(l_3)$.

Furthermore, let $R_2=\nu(r_2)$ and $R_3=\nu(r_3)$ denote $\mathcal{C}(Q_2)\cap\langle U,Q_2\rangle$ and $\mathcal{C}(Q_3)\cap\langle U,Q_3\rangle$ respectively. We have two possibilities, either $r_3=\langle q_1,r_2\rangle\cap l_3$ or $r_3\neq\langle q_1,r_2\rangle\cap l_3$. 

In the first case, we may fix $r_3$ as $(1,1,0)$, and then $\pi$ can be represented by 
$$
\begin{bmatrix} x+z&z&.\\z&y+z&.\\.&.&by+cz \end{bmatrix},
$$
for some $b,c \in \Fq$, where $Q_2=(0,0,0,1,0,b)$ and $Q_3=(1,1,0,1,0,c)$. This case will not define a new orbit, as one can always find a point $Q'_3(x,y,z)$ on the line $\mathcal{Z}(bY+cZ)$, such that $\pi=\langle Q_1,Q_2,Q_3\rangle$ where $Q'_3$ is a  point of rank two with $Q_1 \in \mathcal{C}(Q'_3)$, returning us to Case \ref{case c}. Therefore, we are left with the case that $r_3\neq\langle q_1,r_2\rangle\cap l_3$. Without loss of generality we may choose $r_3=\langle e_1+e_2+e_3\rangle$. Then,  $Q_2=(0,0,0,1,0,b)$ and $Q_3=(1,1,1,1,1,c)$ for some $b,c$ in $\Fq$. It follows that $\pi$ can be represented by $$\pi_{b,c}=\begin{bmatrix} x+z&z&z\\z&y+z&z\\ z&z&by+cz \end{bmatrix},$$ where $b(c-1)\neq 0$ since $Q_2$ and $Q_3$ are points of rank two.

\begin{Lemma}\label{prime}
If $\pi_{b,c}\not\in \Sigma_i$, $1\leq i\leq 14$ and $b(c-1)\neq 0$, then $\pi_{b,c}$ has $q+1$ or $q-1$ points of rank two.
\end{Lemma}
 
 \begin{proof}
 The cubic curve $\mathscr{C}_{b,c}$ associated with $\pi_{b,c}$ is defined by \begin{equation}
 Xf_{b,c}(Y,Z)+g_{b,c}(Y,Z)=0,\end{equation} where \begin{equation}\label{f}f_{b,c}(Y,Z)=bY^2+(b+c)YZ+(1+c)Z^2,\end{equation} and \begin{equation}\label{g}g_{b,c}(Y,Z)=bY^2Z+(1+c)YZ^2,\end{equation} and thus $\pi_{b,c}$ has $q+1$, $q$ or $q-1$ rank-2 points depending on the number of  points of $\mathcal{Z}(f_{b,c})$ on $\PG(1,q)$ being zero, one or two, respectively. By Remark \ref{P121314}, any line in $\pi_{b,c}$ passing through two rank-2 points must belong to $o_{14}$. 
 Therefore, fixing a rank-2 point $Q\in \pi_{b,c}$ and considering all lines spanned by $Q$ and the remaining rank-2 points in $\pi_{b,c}$ partitions the set of rank-2 points in $\pi_{b,c}\setminus\{Q\}$ into pairs. Hence, the number of rank-2 points in $\pi_{b,c}$ is odd. More precisely, $\pi_{b,c}$ has $q+1$ rank-2 points if $\mathcal{Z}(f_{b,c})$ has no points in $\PG(1,q)$ and $q-1$ rank-2 points if $\mathcal{Z}(f_{b,c})$ has two points in $\PG(1,q)$.
 \end{proof}

 \begin{Lemma}\label{q>4}
 Let $q=2^h>2$. If $\pi_{b,c}\not\in \Sigma_i$, $1\leq i\leq 14$ and $b(c-1)\neq 0$, then $q=4$.

 \end{Lemma}
 \begin{proof}

As in the proof of Lemma \ref{prime}, the cubic curve corresponding to
 $\pi_{b,c}$ has equation 
 \begin{equation}
 Xf_{b,c}(Y,Z)+g_{b,c}(Y,Z)=0,
 \end{equation} 
 where $f_{b,c}$ and $g_{b,c}$ are defined in \eqref{f} and \eqref{g}. 
 
 Since $\cZ(f_{b,c})$ has $0$ or $2$ points on $\PG(1,q)$ (which follows from the proof of Lemma \ref{prime}) it follows that $b\neq c$. Consider the line $L=\cZ(X+Z)$ of $\pi_{b,c}$. Since $Q_1$ is not lying on $L$, it follows that $L$ is of type $o_{14}$, $o_{15}$ or $o_{16,2}$, by Remark \ref{P121314} and Table \ref{tableoflines}. More specifically, $L$ has either $3$ points of rank 2 or a unique point of rank 2. In particular, rank-2 points on $L$ satisfy the equation \begin{equation}X^2((1+c)X+(1+b)b)=0,\end{equation} which has exactly two solutions unless $b=1$. 
 
 Similarly, we can consider the line $L'=\cZ(Y+Z)$ of $\pi_{b,c}$. This line has no rank-1 points and has exactly two rank-2 points satisfying \begin{equation}X^2(X+cY)=0,\end{equation} unless $c=0$. Therefore, $b=1$, $c=0$ and $\pi$ reduces to $\pi_{1,0}$, which has $q+1$ rank-2 points if $h$ is odd and $q-1$ rank-2 points if $h$ is even. By Lemma \ref{hessian}, the Hessian of $\mathscr{C}_{1,0}$ defined by \begin{equation}\mathcal{Z}(X(Y^2+YZ+Z^2)+Y^3+YZ^2+ZY^2),\end{equation} intersects $\mathscr{C}_{1,0}$ in three collinear points lying on the line $\mathscr{L}''=\cZ(X+Y)$. Since $Q_1\not\in\mathscr{L}''$ and the configuration of $\pi_{1,0}$ coincides with the second configuration of $\Sigma_{14}$ described in Lemma \ref{phi14}, it follows that for $q>4$, $\pi_{1,0}\in\Sigma_{14}$. Therefore, $\pi_{1,0}$ defines a new orbit if and only if $q=4$. \end{proof}

 We denote this orbit by $\Sigma_{14}'$ which can be represented by $$\Sigma_{14}':\begin{bmatrix} x+z&z&z\\z&y+z&z\\ z&z&y \end{bmatrix}.$$

 \begin{Remark}
 Lemma \ref{q>4} shows that every plane $\pi=\langle Q_1,Q_2,Q_3\rangle$ in $\PG(5,q)$, $q=2^h>4$, containing a unique rank-1 point $Q_1$, where $Q_1\not\in\mathcal{C}(Q_2)\cup\mathcal{C}(Q_3)$ and $\pi\cap \mathcal{N}=\emptyset$, must belong to $\{\Sigma_{12},\Sigma_{13},\Sigma_{14}\}$.
 \end{Remark}

 \begin{Lemma}
 The point-orbit distribution of a plane in $\Sigma_{14}'$ is $[1,0,3,17]$. In particular, $\Sigma_{14}'\not\in  \{\Sigma_1,\Sigma_2,\Sigma_3,\Sigma_4,\Sigma_5,\Sigma_6,\Sigma_7,\Sigma_8,\Sigma_9,\Sigma_{10},\Sigma_{12},\Sigma_{13}\}$.
 \end{Lemma}
 \begin{proof}
 The first part is treated in the proof of Lemma \ref{q>4}. The second part follows from the difference of point-orbit distributions between $\Sigma_{14}'$ and $\Sigma_i$; $1\leq i\leq 12$ and the property that $\Sigma_{14}'$ has three inflexion points by Lemma \ref{q>4} while $\Sigma_{13}$ has none (see Lemma  \ref{inflexion} and Remark \ref{inflexionrmk}).
 \end{proof}

\subsection{Planes containing one rank-1 point and not spanned by points of rank at most 2}\label{sigma14primeorbit}

Let $\pi$ be a plane containing a unique point $Q_1$ of $\mathcal{V}(\Fq)$ and not spanned by points of rank at most $2$. Then, all rank-$2$ points in $\pi$ through $Q_1$ must lie on a unique line (such points exist by Lemma \ref{aux}) and each of the remaining $q$ lines through $Q_1$ must have $q$ points of rank three, and thus belongs to the line-orbit $o_9$ by \cite{lines}. Without loss of generality, let $\pi=\langle Q_1,Q_2,Q_3\rangle$ where $\langle Q_1,Q_3\rangle$ is the representative of $o_9$ in \cite[Table 2]{lines}.
 In particular, take $Q_1(1,0,0,0,0,0)$, $Q_3(0,0,1,1,0,0)$ and $Q_2(0,1,0,a,b,c)$ for some $a,b,c \in \Fq$, then $\pi$ can be represented by
$$\begin{bmatrix} x&y&z\\y&ay+z&by\\z&by&cy \end{bmatrix}.$$

Since points of rank at most two in $\pi$ lie on a line, it follows that the cubic curve $\mathscr{C}=\mathcal{Z}(X(b^2Y^2+acY^2+cYZ)+aYZ^2+cY^3+Z^3)$ associated with $\pi$ is a triple line. Hence, $a=b=c=0$ and the equation of $\mathscr{C}$ reduces to $Z^3=0$. This gives a unique orbit of planes intersecting the Veronese surface in one point and not spanned by points of rank at most two. We denote this orbit by $\Sigma_{15}$, which can be represented by

  $$\Sigma_{15}:\begin{bmatrix} x&y&z\\y&z&.\\z&.&.         \end{bmatrix}.$$

\begin{Lemma}
The point-orbit distribution of a plane in $\Sigma_{15}$ is $[1,1,q-1,q^2]$. In particular, $\Sigma_{15}\neq  \Sigma_i$ for $1\leq i\leq 14$.
\end{Lemma}
\begin{proof}

Clearly, $\Sigma_{15}\neq \Sigma_i$, for all $1\leq i \leq14$, as $\Sigma_{15}$ is not spanned by points of rank at most 2. Let $\pi_{15}$ be the above representative of $\Sigma_{15}$. Points of rank at most $2$ in $\pi_{15}$ correspond to points on the line $\cZ(Z)$, where only the point with homogeneous coordinates $(0,1,0,0,0,0)$ is contained in the nucleus plane which intersects $\pi_{15}$ in $\mathcal{Z}(X,Z)$. Therefore, the point-orbit distribution of a plane in $\Sigma_{15}$ is $[1,1,q-1,q^2]$.
\end{proof}

\subsection{Planes in $\PG(5,2)$}\label{q2planes}
Table \ref{tableplanes} is not completely correct under the action of $ \PGL(3,2)$. In particular, the orbits $\Sigma_1$,...,$\Sigma_{12}$ can be obtained analogously. However, the orbit $\Sigma_{13}$ does not exist for $q=2$. Furthermore, $\Sigma_{14}'$ can no longer be obtained by considering the span of a rank-1 point and a line of type $o_{14}$ as described in Section \ref{sigma14primeorbit}, since no such line exists in this case. More interestingly, planes meeting the Veronese surface non-trivially and not spanned by points of rank at most $2$ split under the action of $\PGL(3,2)$ into $\Sigma_{15}$ and $\Sigma_{15}'$ which is represented by 
$$\Sigma_{15}': \begin{bmatrix} x&y&z\\y&z&.\\z&.&y         \end{bmatrix}.$$

\begin{Remark}\label{q=2planes}
Over the field of two elements, the full setwise stabiliser $K'$ of the Veronese surface is isomorphic to $Sym_7$ (see Remark \ref{remq=2}) which strictly contains $\PGL(3,2)$ and does not preserve the nucleus plane. Under the action of this larger group $K'$ the number of orbits reduces to $5$. Precisely, we have $\Sigma_1=\Sigma_2$, $\Sigma_3=\Sigma_4=\Sigma_5$, $\Sigma_6=\Sigma_{10}$, $\Sigma_7=\Sigma_9=\Sigma_{12}$ and $\Sigma_{8}=\Sigma_{11}=\Sigma_{14}'=\Sigma_{15}=\Sigma_{15}'$, which is easy to verify either by hand or by using the FinInG package in GAP \cite{fining,GAP}.
\end{Remark}

\begin{Theorem}

There are $5$ $K'$-orbits of planes meeting $\cV(\mathbb{F}_2)$ is at least one point, where $K'\cong Sym_7$ is the group stabilising $\cV(\mathbb{F}_2)$. In particular, these orbits split under the action of $\PGL(3,2)$ into $15$ orbits as described in Remark \ref{q=2planes}.
 \end{Theorem}

\section{Comparison with the $q$ odd case}\label{comparison22}

Over finite fields of odd characteristic, there exists a polarity of $\PG(5,q)$ that maps the set of conic planes of $\mathcal{V}(\Fq)$ onto the set of tangent planes of $\mathcal{V}(\Fq)$ (see e.g. \cite[Theorem 4.25]{galois geometry}). This allows the correspondence between rank-1 nets of conics in $\PG(2,q)$, namely, nets with at least one double line, and planes in $\PG(5,q)$ meeting $\cV(\Fq)$ in at least one point, $q$ odd. This correspondence fails over finite fields of characteristic $2$. For instance, let $\pi_6$ be the representative of $\Sigma_6$ defined in Table \ref{tableplanes}. Then, $\pi_6$ meets $\cV(\Fq)$ in a unique point, however its associated net of conics $\mathcal{N}_6$ defined by 
\begin{equation}
\alpha XY+\beta XZ+\gamma (Y^2+cYZ+Z^2)=0
\end{equation}
 has, by Lemma \ref{lem:non-singular}, $q+1$ pairs of real lines defined by the pencil
 $$\mathcal{Z}(XY,XZ),$$ 
 and a unique pair of conjugate imaginary lines given by 
 $$\mathcal{Z}(Y^2+cYZ+Z^2).$$
 Therefore, the hyperplane-orbit distribution of $\pi_6$ is $[0,q+1,1,q^2-1]$, and $\Sigma_6$ has no double lines. In other words, $\mathcal{N}_6$ is not a rank-1 net of conics.

\begin{Corollary}
Rank-1 nets of conics in $\PG(2,q)$ do not correspond to planes having at least one rank-1 point in $\PG(5,q)$ for $q$ even.
\end{Corollary}

 \section{Conclusion}

In this paper, we classified and characterized planes in $\PG(5,q)$ intersecting the Veronese surface $\cV(\Fq)$ in at least one point under the action of the group stabilising $\cV(\Fq)$, $q$ even. This complements the work of the second author, T. Popiel and J. Sheekey \cite{nets} in which such planes were classified for $q$ odd. Our results contributes towards completing the classification of nets of conics in $\PG(2,q)$, $q$ even. In future work, we hope to classify planes in $\PG(5,q)$ which are disjoint from $\cV(\Fq)$. Such a classification would complete the longstanding open problem of classifying nets of conics over finite fields.

\end{document}